\documentclass{amsart}
\usepackage{graphicx}
\usepackage{amssymb,amscd,amsthm,amsxtra}
\usepackage{latexsym}
\usepackage{epsfig}

\newtheorem{thm}{Theorem}[section]
\newtheorem{cor}[thm]{Corollary}
\newtheorem{lem}[thm]{Lemma}
\newtheorem{prop}[thm]{Proposition}
\theoremstyle{definition}
\newtheorem{defn}[thm]{Definition}
\theoremstyle{remark}
\newtheorem{rem}[thm]{Remark}
\numberwithin{equation}{section}

\newcommand{\R}{\mathbb R}

\newcommand{\eps}{\epsilon}

\newcommand{\p}{\partial}
\newcommand{\di}{\displaystyle}

\newcommand{\comment}[1]{}

\begin{document}

\title[A one-phase  free boundary problem for the fractional Laplacian]{Regularity in a one-phase  free boundary problem for the fractional Laplacian}
\author{D. De Silva}
\address{Department of Mathematics, Barnard College, Columbia University, New York, NY 10027}
\email{\tt  desilva@math.columbia.edu}
\author{J.M. Roquejoffre}
\address{Institut de Mathematiques (UMR CNRS 5640), Universit\'e Paul Sabatier, 31062 Toulouse Cedex 4, France}
\email{\tt roque@mip.ups-tlse.fr}


\begin{abstract} For a one-phase free boundary problem involving a fractional Laplacian, we prove that ``flat free boundaries" are $C^{1,\alpha}$. We recover the  
regularity results of Caffarelli for viscosity solutions of the classical Bernoulli-type free boundary problem with the standard Laplacian.
\end{abstract}
\maketitle

\section{Introduction}
The purpose of this paper is to answer a question left open in \cite{CafRS} on the regularity of 
free boundaries for the fractional Laplacian of order $\alpha$ - with $0<\alpha<1$, in the particular case 
$\alpha=1/2$. Here is the setting: consider $g$ a viscosity solution (this notion will be defined 
properly later) of the
following free boundary problem in the  ball $B_1 \subset \R^{n+1}= \R^n \times \R,$ 

\begin{equation}\label{FBintro}\begin{cases}
\Delta g = 0, \quad \textrm{in $B_1^+(g):= B_1 \setminus \{(x,0) : g(x,0) = 0 \} ,$}\\
\dfrac{\p g}{\p U}= 1, \quad \textrm{on  $F(g):=  \p_{\R^n}\{x \in \mathcal{B}_1 : g(x,0)>0\} \cap \mathcal{B}_1,$} 
\end{cases}\end{equation}
where \begin{equation}\label{nabla_U}
\dfrac{\p g}{\p U}(x_0): = \di\lim_{(t,z) \rightarrow (0,0)} \frac{g(x_0+t\nu(x_0),z)} {U(t, z)},  \quad x_0 \in F(g) \end{equation} and  $\mathcal{B}_r \subset \R^n$ is the $n$-dimensional ball of radius $r$ (centered at 0).

The function $U(t,z)$ is the harmonic extension of $\sqrt{t^+}$ to the upper half-plane $\R^{2}_+=\{(t,z) \in \R \times \R, z>0\}$, reflected evenly across $\{z=0\}$. Precisely, 
after the polar change of coordinates $$t= r\cos\theta, \quad  z=r\sin\theta, \quad r\geq 0, -\pi \leq  \theta \leq \pi,$$ $U$ is given by
\begin{equation}\label{U}U(t,z ) = r^{1/2}\cos \frac \theta 2. \end{equation}

One can show that if  a function $g \geq 0$ is harmonic in $B_1^+(g) $ and $F(g)$ is smooth around a point $x_0$ then $\frac{\p g}{\p U}(x_0)$ exists always and it is finite.  
Here $\nu(x_0)$ denotes as usually the normal to $F(g)$ at $x_0$ pointing toward  $\{x : g(x,0)>0\}$.

In this paper, we introduce the notion of viscosity solutions to \eqref{FBintro}  and prove the following result about the regularity of their free boundaries under appropriate flatness assumptions (for all the relevant definitions see Section 2).

\begin{thm} \label{mainT}There exists a universal constant $\bar \eps >0$, such that if $g$ is a viscosity solution to \eqref{FBintro}  satisfying
\begin{equation}\label{Flat1}\|g- U\|_{L^\infty(\overline{B}_1)} \leq \bar \eps, \end{equation} and 
\begin{equation}\label{Flat2} \{x \in \mathcal{B}_1 : x_n \leq -\bar \eps\} \subset \{x \in \mathcal{B}_1 : g(x,0)=0\} \subset \{x \in \mathcal{B}_1 : x_n \leq \bar \eps \},\end{equation} then $F(g)$ is $C^{1,\alpha}$ in $\mathcal{B}_{1/2}$.
\end{thm}

Consequentely $\frac{\p g}{\p U}$ exists and $g$ is a classical solution to \eqref{FBintro}. Moreover, given a point $x_0$ on the free boundary $F(g)$ if one knows that a blow-up sequence of $g$ around $x_0$ ``converges" to the function $U$, then the flatness assumptions \eqref{Flat1}-\eqref{Flat2} are satisfied and hence the free boundary is $C^{1,\alpha}$ around that point. 

Assumption \eqref{Flat1} is a (slightly improved) nondegeneracy assumption which is usually true, and certainly satisfied in the framework of \cite{CafRS}.
In any case it could be removed, but we keep it for simplicity. 

The interest in our free boundary problem \eqref{FBintro} arises from a natural generalization of the following classical Bernoulli-type one-phase free boundary problem:
\begin{equation}\label{AC}\begin{cases}
\Delta u = 0, \quad \textrm{in $\Omega \cap \{u > 0 \} ,$}\\
|\nabla u|= 1, \quad \textrm{on  $\Omega \cap  \p \{u>0\},$} 
\end{cases}\end{equation}
with $\Omega$ a domain in $\R^n.$
A pioneering investigation was that of Alt and Caffarelli \cite{AC} (variational context), and then Caffarelli  \cite{C1,C2,C3} (viscosity solutions context).
See also \cite{CS} for a complete survey.
 
A special class of viscosity solutions to \eqref{FBintro} (with the constant 1 replaced by a precise constant $A$) is provided by minimizers to the energy functional
 $$J(v,B_1) = \int_{B_1} |\nabla v|^2 dx dz + \mathcal{L}_{\R^n} (\{v>0\} \cap \R^n \cap B_1).$$ 
Such minimizers have been investigated by Caffarelli, Sire and the second author in \cite{CafRS}, where general properties (optimal regularity, nondegeneracy,
classification of global solutions), corresponding to those proved by Alt and Caffarelli in \cite{AC} for the Bernoulli-type problem \eqref{AC}, have been obtained.

As for the next issue, i.e. the regularity of the free boundary,
 here is what is proved in \cite{CafRS} in the setting of \eqref{FBintro}:

{\it Let $u(x,y,z)$ be a solution of  \eqref{FBintro} in $B_1 \subset \R^3$. Assume that the free boundary of $u$ is a Lipschitz graph in 
$\mathcal{B}_1$. Then it is a $C^1$ graph in $\mathcal{B}_{1/2}$.}

The idea of this result is that (i) one can find two points on
each side of 0 where the free boundary is flatter than what is dictated by the Lipschitz constant, (ii) this improvement could be propagated inside a small ball 
of controlled size. Thus the three-dimensionality of the problem (or, equivalently, the one-dimensionality of the free boundary) is heavily used. Moreover, this 
argument does not yield the extra H\"older regularity of the derivative - which we believe could itself yield $C^\infty$ regularity of the free boundary. What we propose in this paper is to fill the gap between $C^1$ and $C^{1,\alpha}$, in arbitrary space dimension.

In view of the results in \cite{CafRS}, one knows that the flatness assumptions \eqref{Flat1}, \eqref{Flat2} in our main Theorem  \ref{mainT} are satisfied around each point of the reduced part of the free boundary of a minimizer (see Propositions 4.2 and Theorems 1.2,1.3 in \cite{CafRS}). We thus obtain the following corollary to Theorem \ref{mainT}.

\begin{cor} Let $v$ be a local minimizer to  $$J(v,B_1) = \int_{B_1} |\nabla v|^2 dx dz + \mathcal{L}_{\R^n} (\{v>0\} \cap \R^n \cap B_1).$$ 
 Then the reduced part of the free boundary $F^*(v)$ is $C^{1,\alpha}.$
\end{cor}

Let us now recall how the fractional Laplacian is involved in  \eqref{FBintro}. Consider, for $\alpha\in(0,1)$, the model (which generalizes \eqref{AC})
\begin{equation}\label{ACalpha}\begin{cases}
(-\Delta)^\alpha u = 0, \quad \textrm{in $\Omega \cap \{u > 0 \} ,$}\\
\di\lim_{t \rightarrow 0^+} \dfrac{u
(x_0+t\nu(x_0))}{t^\alpha} = const., \quad \textrm{on  $ \Omega \cap \p\{u>0\},$} 
\end{cases}\end{equation}
with $u$ defined on the whole $\R^n$ with prescribed values outside of $\Omega.$ Recall that, up to a normalization constant
$$(-\Delta)^\alpha u (x) = PV \int_{\R^n} \frac{u(x) - u(y)}{|x-y|^{n+2\alpha}} dy$$
where $PV$ denotes the Cauchy principal value.

When studying local property of the free boundary in \eqref{ACalpha}, the non-locality of the fractional Laplace makes the problem quite delicate.
To avoid this ``contrast" one can make use of an extension property proved by Caffarelli and Silvestre in \cite{CSi} (see for example the work of Caffarelli,  Savin and the second author \cite{CafRSa},  the paper \cite{CafRS}, and the work of Caffarelli, Salsa and Silvestre \cite{CSS} where this strategy has been employed.)
 Precisely, let $u \in C^2(\R^n)$ and let $v$ solve  
 \begin{equation}\label{extform}\begin{cases} - {\mathrm {div}}(z^\beta \nabla v) = 0, \quad \text{in $\R^{n+1}_+= \{(x,z) \in \R^n \times \R, z>0\}$,}\\ v(x,0)=u(x), \quad \text{on $\R^n$,}\end{cases}\end{equation}
 with $\beta=1-2\alpha$.
 Then, 
 \begin{equation}(-\Delta^\alpha)u(x) = -\di\lim_{z \rightarrow 0} (z^\beta v_z(x,z)).\end{equation}
 
 After extending $v$ evenly across the hyperplane $\{z=0\}$, the first equation in \eqref{extform} can be thought in the whole $\R^{n+1}.$
  In view of this formula, the focus shifts on the free boundary problem,
 \begin{equation}\label{FBalpha}\begin{cases}
- {\mathrm{div}}(|z|^\beta \nabla v) = 0, \quad \textrm{in $B_1 \setminus \{(x,0) : v(x,0) = 0 \} ,$}\\
 \di\lim_{(t,z) \rightarrow (0,0)} \dfrac{v(x_0+t\nu(x_0),z)} {U(t, z)} = const., \quad \textrm{on  $\p_{\R^n}\{x  : v(x,0)>0\} \cap \mathcal{B}_1,$} 
\end{cases}\end{equation}
where $U(t,z)$ solves \eqref{extform} in $\R^2$ with $u(t)= (t^+)^\alpha$ and it is extended evenly across $\{z=0\}.$

 For simplicity of exposition we have focused here on the case when $\alpha=1/2$ (in which case the extension formula of \cite{CSi} is a well-know fact). However our result can 
 probably be extended to the general case $\alpha \in (0,1).$

 Our definition of viscosity solution to \eqref{FBintro} is similar to the one introduced by Caffarelli in \cite{C1,C2} to deal with the problem \eqref{AC}. Indeed our result generalizes to this non-local setting the ``flatness implies regularity" theory  developed by Caffarelli in \cite{C2}. Let us also mention that Theorem  \ref{mainT} is probably optimal. Indeed, quite similarly to what happens for minimal surfaces,
singular free boundaries for the Bernoulli-type problem  \eqref{AC} were discovered by Jerison and the first author \cite{DSJ}.

 Let us now describe our strategy to obtain Theorem  \ref{mainT}. The main idea to prove Theorem  \ref{mainT} is to show that 
$F(g)$ enjoys an ``improvement of flatness" property, that is if $F(g)$ oscillates $\eps$ away from a hyperplane in $\mathcal{B}_1$ ($\eps$ small), then
in $\mathcal{B}_{\rho}$ ($\rho$ universal) it oscillates $\eps \rho/2$ away from possibly a
different hyperplane.  To obtain this improvement of flatness, we use a compactness argument which goes as follows: assume one cannot do it
however flat the free boundary is, then we blow it up it in the $x_n$ direction, thus linearizing the problem into a limiting one, for which we prove that improvement of flatness holds
- thus   a contradiction. This scheme was used by Savin \cite{S} to prove regularity of small solutions of fully nonlinear equations - including an elegant proof of the De Giorgi theorem for minimal surfaces. The key tool is a geometric Harnack inequality that localizes the free boundary well, and allows
the passage to the limit under rescalings. 

Such compactness arguments can also be found in Wang  \cite{W} for the regularity of the solutions of $p$-Laplace equations.  More recently, the first author followed this strategy in \cite{D} to provide a new proof of the Caffarelli ``flat implies smooth'' theory. The scheme is the same here, up to the fact that the construction of the sub-solution
opening the way to the Harnack inequality is different to that of \cite{D}, and that the linear problem obtained eventually is non-standard (and interesting in its own).

The paper is organized as follows. In Section 2 we introduce the notion of viscosity solutions to \eqref{FBintro} and we prove a basic Comparison Principle for such solutions. In Section 3 we explain how to interpret our solutions as perturbations of $U$ after a ``domain variation" in the $e_n$-direction and we present basic facts about such domain variations. Throughout the paper, this will be a convenient way of thinking about our viscosity solutions. 
In Section 4 we describe the linear problem associated to \eqref{FBintro} and later in Section 8 we obtain a regularity result for its solutions. Section 5 contains some technical lemmas leading to the proof of Harnack inequality. In Section 6 we exhibit the proof of Harnack inequality using the barrier which we will construct later in the Appendix. Finally in Section 7 we provide the proof of the ``improvement of flatness"  property.  
\

{\bf Acknowledgement.} J.-M. R. is supported by the ANR grant PREFERED.

\section{Definitions and basic lemmas}

In this Section we introduce notation and definitions which we will use throughout the paper and we prove a standard basic lemma (Comparison Principle).
 
A point $X \in \R^{n+1}$ will be denoted by $X= (x,z) \in \R^n \times \R$. We will also use the notation $x=(x',x_n)$ with $x'=(x_1,\ldots, x_{n-1}).$ A ball in $\R^{n+1}$ with radius $r$ and center $X$ is denoted by $B_r(X)$ and for simplicity $B_r = B_r(0)$. Also we use $\mathcal{B}_r$ to denote the $n$-dimensional ball $B_r \cap \{z=0\}$. 

Let $v(X)$ be a continuous non-negative function in $B_1$. We associate to $v$ the following sets: \begin{align*}
& B_1^+(v) := B_1 \setminus \{(x,0) : v(x,0) = 0 \} \subset \R^{n+1};\\
& \mathcal{B}_1^+(v):= B_1^+(v) \cap \mathcal{B}_1 \subset \R^{n};\\
& F(v) := \p_{\R^n}\mathcal{B}_1^+(v)\cap \mathcal{B}_1 \subset \R^{n};\\
& \mathcal{B}_1^0(v) : =Int_{\R^n}\{x \in \R^n : v(x,0)=0\} \subset \R^{n}.
\end{align*}  Often subsets of $\R^n$ are embedded in $\R^{n+1}$, as it will be clear from the context. 


We may refer to $\mathcal{B}_1^0(v)$  as to the zero plate of $v$, while  $F(v)$ is called the free boundary of $v$.

We consider the free boundary problem

\begin{equation}\label{FB}\begin{cases}
\Delta g = 0, \quad \textrm{in $B_1^+(g) ,$}\\
\dfrac{\p g}{\p U}= 1, \quad \textrm{on $F(g)$}, 
\end{cases}\end{equation}
where 
$$\dfrac{\p g}{\p U}(x_0):=\di\lim_{(t,z) \rightarrow (0,0)} \frac{g(x_0+t\nu(x_0),z)} {U(t, z)} , \quad \textrm{$X_0=(x_0,0) \in F(g)$}.$$ 
Here $\nu(x_0)$ denotes the unit normal to $F(g)$ at $x_0$ pointing toward $\mathcal{B}_1^+(g)$ and $U$ is the function defined in \eqref{U}. Also, throughout the paper
we call $U(X):=U(x_n,z).$ 



We now introduce the notion of viscosity solutions to \eqref{FB}. 
First we need the following standard notion.

\begin{defn}Given $g, v$ continuous, we say that $v$
touches $g$ by below (resp. above) at $X_0 \in B_1$ if $g(X_0)=
v(X_0),$ and
$$g(X) \geq v(X) \quad (\text{resp. $g(X) \leq
v(X)$}) \quad \text{in a neighborhood $O$ of $X_0$.}$$ If
this inequality is strict in $O \setminus \{X_0\}$, we say that
$v$ touches $g$ strictly by below (resp. above).
\end{defn}

\begin{defn}\label{defsub} We say that $v \in C^2(B_1)$ is a (strict) comparison subsolution to \eqref{FB} if $v$ is a  non-negative function in $B_1$ which is even with respect to $z=0$ and it satisfies
\begin{enumerate} \item $\Delta v \geq 0$ \quad in $B_1^+(v)$;\\
\item $F(v)$ is $C^2$ and if $x_0 \in F(v)$ we have
$$v (x,z) = \alpha U((x-x_0) \cdot \nu(x_0), z)+ o(|(x-x_0,z)|^{1/2}), \quad \textrm{as $(x,z) \rightarrow (x_0,0),$}$$ with $$\alpha \geq 1,$$ where $\nu(x_0)$ denotes the unit normal at $x_0$ to $F(v)$ pointing toward $\mathcal{B}_1^+(v);$\\
\item Either $v$ is not harmonic in $B_1^+(v)$ or $\alpha >1.$
\end{enumerate} 
\end{defn}

Similarly one can define a (strict) comparison supersolution. 

\begin{defn}We say that $g$ is a viscosity solution to \eqref{FB} if $g$ is a  continuous non-negative function in $B_1$ which is even with respect to $z=0$ and it satisfies
\begin{enumerate} \item $\Delta g = 0$ \quad in $B_1^+(g)$;\\ \item Any (strict) comparison subsolution (resp. supersolution) cannot touch $g$ by below (resp. by above) at a point $X_0 = (x_0,0)\in F(g). $\end{enumerate}\end{defn}

\begin{rem} By standard arguments, if $g$ is a viscosity solution to \eqref{FB} and $F(g)$ is $C^1$ then $g$ is a classical solution of the free boundary problem (see for example Proposition 4.2 in \cite{CafRS}.) 
Moreover, as remarked in the Introduction one can show that given any continuous function  $g$ which is harmonic in $B_1^+(g) ,$ 
then $\frac{\p g}{\p U} (x_0)$  exists at each point around which $F(g)$ is $C^{1,\alpha}$. These facts motivate our problem and the definition of viscosity solution.
\end{rem}

\begin{rem}\label{rescale} We remark that if $g$ is a viscosity solution to \eqref{FB} in $B_\rho$, then $$g_{\rho}(X) = \rho^{-1/2} g(\rho X), \quad X \in B_1$$ is a viscosity solution to \eqref{FB} in $B_1.$
\end{rem}

We finish this section by stating and proving a comparison principle for problem \eqref{FB} which will be a key tool in the proof of Harnack inequality in Section 7.

\begin{lem}[Comparison principle] Let $g, v_t \in C(\overline{B}_1)$ be respectively a solution and a family of  subsolutions to \eqref{FB}, $t \in [0,1]$. Assume that
\begin{enumerate}
\item $v_0 \leq g,$ in $\overline{B}_1;$
\item $v_t \leq g$ on $\p B_1$ for all $t \in [0,1];$
\item $v_t < g$ on $\mathcal{F}(v_t)$ which is the boundary in $\p B_1$ of the set $\p \mathcal{B}_1^+(v_t) \cap \p \mathcal{B}_1$, for all $t\in [0,1];$
\item $v_t(x)$ is continuous in $(x,t) \in \overline{B}_1 \times [0,1]$ and $\overline{\mathcal{B}_1^+(v_t)}$ is continuous in the Hausdorff metric.
\end{enumerate}
Then 
\begin{equation*} v_t \leq g \quad \text{in $\overline{B}_1$, for all $t\in[0,1]$.}
\end{equation*}
\end{lem}\begin{proof}
Let $$A:= \{t \in [0,1] : v_t(x) \leq g(x) \ \ \text{on $\overline{B}_1$} \}.$$ In view of (i) and (iv) $A$ is closed and non-empty. Our claim will follow if we show that $A$ is open. 
Let $t_0 \in A$, then $v_{t_0} \leq g$ on $\overline{B}_1$ and by the definition of viscosity solution
$$F(v_{t_0}) \cap F(g) = \emptyset.$$ Together with (iii) this implies that 
$${\mathcal{B}_1^+(v_{t_0})} \subset  \mathcal{B}^+_1(g), \quad F(v_{t_0}) \cup \mathcal{F}(v_{t_0})  \subset  \{x \in \overline{\mathcal{B}}_1 : g(x,0)>0\}.$$ By (iv) this gives that for $t$ close to $t_0$
\begin{equation}\label{inclus}
{\mathcal{B}_1^+(v_{t})} \subset  \mathcal{B}^+_1(g), \quad F(v_{t}) \cup \mathcal{F}(v_{t})  \subset  \{x \in \overline{\mathcal{B}}_1 : g(x,0)>0\}.
\end{equation} 
Call $D:= B_1 \setminus (\mathcal{B}^0_1(v_t) \cup F(v_t)) .$
Combining \eqref{inclus} with assumption (ii) we get that 
$$v_t \leq g \quad \text{on $\p D,$}$$ and by the maximum principle the inequality holds also in $D$. Hence 
$$v_t \leq g \quad \text{in $\overline{B}_1,$}$$ and $t \in A$ which shows that $A$ is open.
\end{proof}

\begin{cor} \label{compmon}Let $g$ be a solution to \eqref{FB} and let $v$ be a  subsolution to \eqref{FB} in $B_2$ which is strictly monotone increasing in the $e_n$-direction in $B_2^+(v)$. Call
$$v_t(X):=v(X+ t e_n), \quad X \in B_1.$$
Assume that for $-1 \leq t_0 < t_1\leq 1$
$$v_{t_0} \leq g, \quad \text{in $\overline{B}_1,$}$$ and
$$v_{t_1} \leq g \quad  \text{on $\p B_1,$}  \quad v_{t_1} < g  \quad \text{on $\mathcal{F}(v_{t_1}).$}$$
Then 
\begin{equation*} v_{t_1} \leq g \quad \text{in $\overline{B}_1$.}
\end{equation*}

\end{cor}



\section{The function $\tilde g$} 

Let $g$ be a viscosity solution to \eqref{FB}. Throughout the paper, it will be convenient to interpret $g$ as a perturbation of $U$ via a domain variation in the $e_n$-direction. In this section we explain some basic facts about such domain variations. 

Let $\eps>0$ and let $g$ be a  continuous non-negative function in $\overline{B}_\rho$. 
Here and henceforth we denote by $P$ the half-hyperplane $P:= \{X \in \R^{n+1} : x_n \leq 0, z=0\}$ and by $L:= \{X \in \R^{n+1}: x_n=0, z=0\}$.  
To each $X \in \R^{n+1} \setminus P$ we associate $\tilde g_\eps(X) \subset \R$ via the formula 
\begin{equation}\label{deftilde} U(X) = g(X - \eps w e_n), \quad \forall w \in \tilde g_\eps(X).\end{equation} 

We sometimes call $\tilde g_\eps$  the $\eps$- domain variation associated to $g$. 
By abuse of notation, from now on  we write $ \tilde g_\eps(X)$ to denote any of the values in this set.

If g satisfies\begin{equation}\label{flattilde}U(X - \eps e_n) \leq g(X) \leq U(X+\eps e_n) \quad \textrm{in $B_\rho,$}\end{equation} then $$\tilde g_\eps (X) \in [-1,1].$$

Indeed call $$Y= X - \eps \tilde g_\eps (X)e_n,\quad X \in \R^{n+1} \setminus P.$$
Then according to \eqref{flattilde},
$$U(Y-\eps e_n) \leq g(Y)= U(Y+\eps \tilde g_\eps (X) e_n) \leq U(Y+\eps e_n).$$
Since $U(Y+\eps \tilde g_\eps (X) e_n) =U(X) >0$ our claim  follows from the strict monotonicity of $U$ in the $e_n$-direction (outside of $P$.)

Moreover, under the assumption \eqref{flattilde} for each $X \in B_{\rho-\eps} \setminus P$ there exists at least one value $\tilde{g}_\eps(X)$ such that 
\begin{equation} \label{til}U(X) = g(X - \eps \tilde{g}_\eps(X)e_n).\end{equation} 
Indeed, it follows from \eqref{flattilde} that 
$$g(X- \eps e_n) \leq U(X) \leq g(X+\eps e_n), \quad X \in B_{\rho-\eps}$$
and our claim follows by the continuity of $g(X-\delta \eps e_n), \delta \in [-1,1].$

Thus if \eqref{flattilde} holds, for all $\eps >0$  we can associate to $g$ a possibly multi-valued function $\tilde{g}_\eps$ defined at least on $B_{\rho-\eps} \setminus P$ and taking values in $[-1,1]$ which satisfies \eqref{til}. Moreover if $g$ is strictly monotone  in the $e_n$-direction in $B^+_\rho(g)$, then $\tilde{g}_\eps$ is single-valued. 

The following elementary lemma will be used to obtain a useful comparison principle for the $\eps$-domain variations of solutions to \eqref{FB}.

\begin{lem}\label{elem} Let $g, v$ be non-negative continuous functions in $B_\rho$. Assume that $g$ satisfies the flatness condition \eqref{flattilde} in $B_\rho$ and that $v$ is strictly increasing in the $e_n$-direction in $B_\rho^+(v).$ Then if 
$$v \leq g \quad \text{in $B_\rho,$}$$ and 
$\tilde v_\eps$ exists on $B_{\rho-\eps} \setminus P$ we have that
$$\tilde v_\eps \leq \tilde g_\eps \quad \text{on  $B_{\rho-\eps} \setminus P.$}$$
Viceversa, if $\tilde v_\eps$ exists on $B_s \setminus P$ and 
$$\tilde v_\eps \leq \tilde g_\eps \quad \text{on $B_s \setminus P,$}$$ then
$$v \leq g \quad \text{on $B_{s-\eps}$}.$$
\end{lem}
\begin{proof} The first implication is obvious. Indeed, assume by contradiction that $v \leq g$ in $B_\rho$ and there exists $X \in B_{\rho -\eps} \setminus P$ such that $$\tilde v_\eps (X) > \tilde g_\eps(X).$$
By the strict monotonicity of $v$ in the $e_n$-direction in $B_\rho^+(v)$ we have that
$$0 < U(X) = g(X - \eps \tilde g_\eps (X)e_n) =  v(X - \eps \tilde v_\eps (X)e_n) < v( X - \eps \tilde g_\eps (X)e_n).$$ Thus there exists $Y= X - \eps \tilde g_\eps (X)e_n \in B_\rho$ such that $g(Y) < v(Y),$ 
a contradiction.
Viceversa, suppose that $\tilde v_\eps \leq \tilde g_\eps$ in $B_{s} \setminus P$. 
For a fixed $Y \in B_{s -\eps}$ we know by the flatness assumption \eqref{flattilde} that
$$U(Y-\eps e_n) \leq g(Y) \leq U(Y+\eps e_n).$$
Thus, there exists $X \in B_{s}$ with $x_i = y_i$ for $ i\neq n$ and $x_n \in [y_n - \eps e_n, y_n+\eps e_n]$ such that 
$$g(Y) = U(X).$$ Suppose $g(Y) \neq 0$, then the identity above means that one of the possible values of $\tilde g_\eps (X) = \frac{y_n-x_n}{\eps}$. Again, using that $v$ is increasing in the $e_n$-direction we get:
$$g(Y)=U(X)=v(X-\eps \tilde v_\eps(X)e_n) \geq v(X-\eps \tilde g_\eps(X)e_n) = v(Y), \quad Y \in B_s^+(g). $$ Thus the desired inequality holds in $B_s^+(g)$ and hence by continuity it holds in the full ball $B_s$. 
\end{proof}




We now state and prove the desired comparison principle, which will follow immediately from the Lemma above and Corollary  \ref{compmon}.

\begin{lem}\label{linearcomp}
Let $g, v$ be respectively a solution and a subsolution to \eqref{FB} in $B_2$, with $v$ strictly increasing in the $e_n$-direction in $B_2^+(v).$ Assume that $g$ satisfies the flatness assumption \eqref{flattilde} in $B_2$
for $\eps>0$ small  and that $\tilde v_\eps$ exists in $B_{2-\eps} \setminus P$ and satisfies  $$|\tilde v_\eps| \leq C.$$
If,
\begin{equation}\label{start}
 \tilde v_\eps + c \leq  \tilde g_\eps \quad \text{in $(B_{3/2} \setminus \overline{B}_{1/2}) \setminus P,$} 
\end{equation} then 
\begin{equation}\label{conclusion}
 \tilde v_\eps + c \leq \tilde g_\eps  \quad \text{in $B_{3/2} \setminus P.$} 
\end{equation} 
\end{lem}
\begin{proof} We wish to apply Corollary \ref{compmon} to the functions $g$ and $$v_{\eps t} = v(X+ \eps t e_n).$$

We need to verify that for some  $t_0 < t_1=c$
\begin{equation}\label{n1}v_{\eps t_0} \leq g \quad \text{in $\overline{B}_1,$}\end{equation} and for all $\delta>0$ and small
\begin{equation}\label{n2}v_{\eps t_1} \leq g \quad  \text{on $\p B_1,$}  \quad v_{\eps t_1} < g  \quad \text{on $\mathcal{F}(v_{\eps(t_1-\delta)}).$}\end{equation}
Then our Corollary  implies
\begin{equation*} v_{\eps(t_1-\delta)} \leq g \quad \text{in $\overline{B}_1$.} \end{equation*}
By letting $\delta$ go to 0, we obtain that 
\begin{equation*} v_{\eps t_1} \leq g \quad \text{in $\overline{B}_1$,} \end{equation*} which in view of Lemma \ref{elem} gives 
\begin{equation*}\widetilde{(v_{\eps t_1})}_\eps  \leq \tilde g_\eps \quad \text{in $B_{1-\eps}\setminus P$,}\end{equation*} assuming that the $\eps$-domain variation on the left hand side exists on $B_{1-\eps} \setminus P.$ On the other hand, it is easy to verify that on such set
\begin{equation}\label{tildetrans}\widetilde{(v_{\eps t})}_\eps (X) = \tilde v_\eps (X) + t,\end{equation} and hence we have 
\begin{equation*}\tilde v_\eps + c =\tilde v_\eps + t_1 \leq \tilde g_\eps \quad \text{in $B_{1-\eps} \setminus P,$} \end{equation*} which gives the desired conclusion.
We are left with the proof of \eqref{n1}-\eqref{n2}.

In view of Lemma \ref{elem}, in order to obtain \eqref{n1} it suffices to show that
$$\widetilde{(v_{\eps t_0})}_\eps  \leq \tilde g_\eps, \quad \text{in $B_{1+\eps} \setminus P,$}$$
which by \eqref{tildetrans} becomes
$$\tilde v_\eps + t_0 \leq \tilde g_\eps, \quad \text{in $B_{1+\eps} \setminus P.$}$$
This last inequality holds trivially since $\tilde g_\eps$ and $\tilde v_\eps$ are bounded.

For \eqref{n2}, notice that the first inequality follows easily from our assumption \eqref{start} together with \eqref{tildetrans} and Lemma \ref{elem}.  More precisely we have that

$$v_{\eps t_1} \leq g \quad \text{in $B_{\frac 3 2-\eps} \setminus B_{\frac 1 2+\eps}$.}$$

In particular, from the strict monotonicity of $v$ in the $e_n$-direction in $B_2^+(v)$ we have that
$$v_{\eps t_1} > 0 \quad \text{on $\mathcal{F}(v_{\eps(t_1-\delta)}),$}$$
which combined with the previous inequality gives that
$$g >0 \quad \text{on $\mathcal{F}(v_{\eps(t_1-\delta)}),$}$$
that is the second condition in \eqref{n2}. 
\end{proof}

Finally, given $\eps>0$ small and a Lipschitz function $\tilde{\varphi}$ defined on $B_{\rho}(\bar X)$,  with values in $[-1,1]$, then there exists a unique function $\varphi_\eps$ defined at least on $B_{\rho-\eps}(\bar X)$ such that \begin{equation}\label{deftilde2} U(X) = \varphi_\eps(X - \eps \tilde{\varphi}(X)e_n), \quad X \in B_\rho(\bar X).\end{equation}
Moreover such function $\varphi_\eps$ is increasing in the $e_n$-direction. 

With a similar argument as in Lemma \ref{elem} we can conclude that if
 $g$ satisfies the flatness assumption \eqref{flattilde} in $B_1$ and $\tilde{\varphi}$ is as above then (say $\rho,\eps<1/4$, $\bar X \in B_{1/2},$)
\begin{equation}\label{gtildeg}\tilde \varphi \leq \tilde g_\eps \quad \text{in $B_\rho(\bar X) \setminus P$} \Rightarrow \varphi_\eps \leq g \quad \text{in $B_{\rho -\eps}(\bar X)$}.\end{equation}

We will use this fact in the proof of our improvement of flatness theorem.

\section{The linearized problem.}  

We introduce here the linearized problem associated to \eqref{FB}.  Here and later $U_n$ denotes the $x_n$-derivative of the function $U$ defined in \eqref{U}. Recall also that we denote by $P$ the half-hyperplane $P:= \{X \in \R^{n+1} : x_n \leq 0, z=0\}$ and by $L:= \{X \in \R^{n+1}: x_n=0, z=0\}$.



Given  $w \in C(B_1)$  and  $X_0=(x'_0,0,0) \in B_1 \cap L,$ we call
$$|\nabla_r w |(X_0) := \di\lim_{(x_n,z)\rightarrow (0,0)} \frac{w(x'_0,x_n, z) - w (x'_0,0,0)}{r}, \quad   r^2=x_n^2+z^2 .$$
Once the change of unknowns  \eqref{deftilde} has been done, the linearized problem associated to \eqref{FB} is 
\begin{equation}\label{linear}\begin{cases} \Delta (U_n w) = 0, \quad \text{in $B_1 \setminus P,$}\\ |\nabla_r w|=0, \quad \text{on $B_1\cap L$.}\end{cases}\end{equation}

As we will show later in Section 8, if $w \in C(B_1)$  satisfies 
$$\Delta (U_n w) = 0 \quad \text{in $B_1 \setminus P,$}$$
$w$ is even with respect to $\{z=0\}$, and $w$ is smooth in the $x'$-direction, then given $X_0=(x'_0,0,0) \in B_1 \cap L,$ \begin{equation}\label{exp} w(X) = w(X_0) + a \cdot (x' - x'_0) + b r + O(|x'-x'_0|^2 + r^{3/2}),  \end{equation} with $a \in \R^{n-1}, b \in \R$ depending on $X_0$.

This motivates our notion of viscosity solution for this problem which we define below. 




\begin{defn}\label{linearsol}We say that $w$ is a solution to \eqref{linear}  if $w \in C(B_1)$, $w$ is even with respect to $\{z=0\}$ and it satisfies
\begin{enumerate}\item $\Delta (U_n w) = 0$ \quad in $B_1 \setminus P$;\\ \item Let $\phi$ be continuous around  $X_0=(x'_0,0,0) \in B_1 \cap L$ and satisfy $$\phi(X) = \phi(X_0) + a(X_0)\cdot (x' - x'_0) + b(X_0) r + O(|x'-x'_0|^2 + r^{3/2}), $$ with $$b(X_0) \neq 0.$$  If $b(X_0) >0$ then  $\phi$ cannot touch $w$ by below at $X_0$,  and if $b(X_0)< 0$ then $\phi$ cannot touch $w$ by above at $X_0$. \end{enumerate}\end{defn}



In Section 8, we will show the following main regularity result about viscosity solutions to \eqref{linear}.

\begin{thm}[Improvement of flatness]\label{lineimpflat3} There exists a universal constant $C$ such that if $w$ is a viscosity solution to \eqref{linear2} in $B_1$ with $$-1 \leq w(X) \leq 1\quad \text{in $B_1,$}$$ then  \begin{equation*}  a_0 \cdot x' -C|X|^{3/2} \leq w(X)-w(0) \leq a_0 \cdot x' + C |X|^{3/2},\end{equation*}
for some vector $a_0 \in R^{n-1.}$\end{thm}

We conclude this short section with a remark which   we will use in the proof of the theorem above.

\begin{lem}\label{rk} Let $w_1, w_2 \in C(B_1)$ satisfy
$$\Delta (U_n w_i) = 0, \quad \text{in $B_1 \setminus P$, i=1,2.}$$ Then $w_1$ and $w_2$ cannot touch  (either by above or below) on $P \setminus L$, unless they coincide.
\end{lem}
\begin{proof}
Assume by contradiction that
$$w_1(X_0) = w_2(X_0), \quad X_0 \in P \setminus L,$$
and $$w_1 \geq w_2, \quad \text{in $B_\rho(X_0)$}.$$ Then $U_n (w_1-w_2)$ is a non-negative harmonic function in $B_1 \setminus P$ which vanishes continuously on $P \setminus L$. Hence unless $w_1=w_2$, by the boundary Harnack inequality (in the appropriate domain), 
$$U_n(w_1 - w_2) \geq \delta U_n \quad \text {in $B_{\rho/2}(X_0) \cap \{z >0\},$}$$ for some small positive constant $\delta.$ Thus  $$w_1 - w_2 \geq \delta \quad \text {in $B_{\rho/2}(X_0) \cap \{z >0\},$}$$ and by continuity $$(w_1-w_2)(X_0) >0,$$ a contradiction.
\end{proof}

\section{Properties of $U$.}

The first two lemmas in this Section describe properties of $U$ which will be used in the proof of Harnack inequality, and in particular when constructing the barriers which are used in that proof.

The third lemma, which is incorporated here since its proof uses similar arguments to the proof of the first two lemmas, allows us to replace the assumptions in our main Theorem \ref{mainT} with a more standard ``flatness"  assumption of the form $$U(X-\eps e_n) \leq g(X) \leq  U(X+\eps e_n), \quad \text{in $B_{1}$}.$$

\begin{lem}\label{basic} Let $g \in C(B_2)$, $g \geq 0$ be a harmonic function in $B_2^+(g)$ and let $\bar X = \frac 3 2 e_n.$ Assume that 
$$g \geq U \quad \text{in $B_2$}, \quad g(\bar X) - U(\bar X) \geq \delta_0$$ for some $\delta_0>0$, then
\begin{equation}\label{gU}g \geq (1+c\delta_0) U \quad \text{in $B_{1}$}\end{equation} for a small universal constant $c$.
In particular, for any $0 < \eps < 2$  
\begin{equation}\label{cor1}U(X + \eps e_n) \geq (1+c\eps)U(X) \quad \text{in $B_1$},\end{equation} with $c$ small universal.
\end{lem}
\begin{proof}

 Call $g^*$ the harmonic function in $$D=B_{3/2} \setminus \{x \in \mathcal{B}_{3/2} : x_n \leq 0\},$$ such that
$$g^*= g \quad \text{on $\p B_{3/2}$}, \quad g^*= 0 \quad \text{on $\{x \in \mathcal{B}_{3/2} : x_n \leq 0\}. $}$$ Then by the maximum principle $$g \geq g^* \quad \text{on $\overline{B}_{3/2},$}$$ and it suffices to show that  \eqref{gU} holds with $g^*$ on the left hand side.

Since $g \geq U$ in $B_2$ we have $$g^* - U = g - U \geq 0 \quad \text{on $\p B_{3/2}$}, \quad g^*-U=0 \quad \text{on  $\{x \in \mathcal{B}_{3/2} : x_n \leq 0\},$} $$ and hence  $g^* - U\geq 0$ in $D$ where it is also harmonic. Moreover, from the assumption $g(\bar X) - U(\bar X) \geq \delta_0$ we get by Harnack inequality that
$$g^* - U = g - U \geq  c_0 \delta_0 \quad \text{on $\p B_{3/2} \cap B_{1/4}(\bar X)$}.$$ Thus
$$g^*( \tilde X) - U (\tilde X) \geq  c_1 \delta_0, \quad \text{at some $\tilde X  \in B_1 \cap D.$}$$
Thus, by the boundary Harnack inequality we get that for $c>0$ universal, $$g^* - U \geq c_2 \frac{g^*(\tilde X) - U(\tilde X)}{U(\tilde X)} U \geq c  \delta_0 U \quad \text{in $B_1$},$$
as desired.

In particular, if $g(X) = U(X+\eps e_n)$ the assumptions of the lemma are satisfied. Indeed $U$ is monotone increasing in the $e_n$-direction thus $U(X+\eps e_n) \geq U(X)$ in $B_2$. Moreover,
$$U(\bar X + \eps e_n) - U(\bar X) = U_t(\bar X + \lambda e_n) \eps \geq c' \eps, \quad \lambda \in (0, \eps),$$ with $c'$ universal. 
\end{proof}


\begin{lem} \label{cor2} For any $\eps >0 $ small, given $2\eps < \bar \delta <1$, there exists a constant $C>0$ depending on $\bar \delta$ such that   
$$U(t + \eps, z) \leq (1+C\eps)U(t,z) \quad \text{in $\overline{B}_1 \setminus B_{\bar \delta} \subset \R^2$}.$$
\end{lem}
\begin{proof} In this Lemma $B_\rho$ denotes a ball of radius $\rho$ in $\R^2.$

Since $U$ is monotone increasing in the $t$-direction, $U(t+\eps , z) - U(t,z)$ is non-negative and harmonic in the set $D_{\bar \delta} := (B_2  \setminus \{t \in (-2,2) : t \leq 0\}) \setminus \overline{B}_{\bar \delta/2}.$ Moreover, 
$$U(3/2 + \eps, 0) - U(3/2,0) = U_t(t, 0) \eps \leq C_0 \eps, \quad t \in (3/2, 3/2+\eps),$$ with $C_0$ universal. 
By the boundary Harnack inequality in $D_{\bar \delta}$,
$$U(t+\eps, z) - U(t,z) \leq C_1 \frac{U(3/2+\eps, 0 ) - U(3/2,0)}{U(3/2,0)} U(t,z) \leq C \eps U(t,z) \quad \text{in $\overline{B}_1 \setminus B_{\bar \delta}$},$$
as desired.
\end{proof}

 \begin{lem}\label{hyp} Let $g \in C(\overline{B}_2)$, $g \geq 0$ be a harmonic function in $B_2^+(g)$ satisfying
\begin{equation}\label{Linfty}\|g- U\|_{L^\infty(\overline{B}_2)} \leq \delta, \end{equation} and 
\begin{equation*} \{x \in \mathcal{B}_2 : x_n \leq -\delta\} \subset  \{x \in \mathcal{B}_2 : g(x,0)=0\} \subset \{x \in \mathcal{B}_2 : x_n \leq \delta\},\end{equation*} with $\delta >0$ small universal. Then 
\begin{equation}\label{flat2}U(X-\eps e_n) \leq g(X) \leq  U(X+\eps e_n) \quad \text{in $B_{1},$} \end{equation} for some $\eps = K \delta,$ $K$ universal.
\end{lem}
\begin{proof} Let $\bar g$ be the harmonic function in $$B_2^\delta:= B_2 \setminus \{x  \in \mathcal{B}_2 : x_n \leq  -\delta\},$$ such that $$\bar g = g \quad \text{on $\p B_2$}, \quad \bar g =0 \quad \text{on $\{x \in \mathcal{B}_2 : x_n \leq -\delta\}$}.$$ Since $g$ is subharmonic in $B_2^\delta $, the maximum principle gives us $$g \leq \bar g \quad \text{on $\overline{B}_2$}.$$ We need to show that for $K>0$ universal, \begin{equation}\label{claim}\bar g \leq U(X + K\delta e_n) \quad \text{in $B_1$}.\end{equation} (The lower bound follows from a similar argument). Since $U$ is monotone increasing in the $e_n$-direction and it satisfies \eqref{Linfty} we get that
$$U(X+ \delta e_n) \geq U(X) \geq g(X) - \delta \quad \text{on $\overline{B}_2,$}$$ and hence 
$$U(X+\delta e_n) \geq \bar g (X) -\delta \quad \text{on $\p B_2$}.$$ By the maximum principle in the domain $B_2^\delta$ we get that this inequality holds in $B_2$ and hence \begin{equation}\label{BB}\bar g (X) - U(X+\delta e_n) \leq \delta \quad \text{in $B_2$}.\end{equation} Let $g^*$ be the harmonic function in $B_{3/2} \setminus \{x\in \mathcal{B}_{3/2}: x_n \leq -\delta\}$ such that
$$g^* = \delta \quad \text{on $\p B_{3/2}$}, \quad g^*= 0 \quad \text{on $\{x \in \mathcal{B}_{3/2} : x_n \leq -\delta\}$}.$$ Clearly $$0 \leq g^* \leq \delta.$$ Then by the boundary Harnack inequality, say for $\bar X = e_n$
\begin{equation}\label{Hb}g^*(X) \leq \bar C \frac{g^*(\bar X)}{U(\bar X + \delta e_n)} U(X+\delta e_n) \leq C \delta U(X+\delta e_n) \quad \text{in $B_1$}, \end{equation} with $C>0$ universal.
Moreover, in view of \eqref{BB} again by the maximum principle we have $$\bar  g (X) - U(X+\delta e_n) \leq g^*(X) \quad \text{in $B_{3/2}$}.$$ This inequality together with \eqref{Hb} gives that
$$\bar g (X) \leq (1+C\delta) U(X + \delta e_n) \quad \text{in $B_1$.}$$ By  \eqref{cor1} (applied to a translate of $U$) we have that for $K>1$ 
$$(1+C\delta)U(X + \delta e_n) \leq \frac{1+C\delta}{1+cK\delta} U(X+K\delta e_n) \leq U(X+K\delta e_n) \quad \text{in $B_{1}$},$$ as long as $K$ is large enough. Combining these two last inequalities we obtain the desired claim \eqref{claim}.\end{proof}

\section{Harnack Inequality}

In this Section we state and prove a Harnack type inequality for solutions to our free boundary problem \eqref{FB}.

\begin{thm}[Harnack inequality]\label{mainH} There exists $\bar \eps > 0$  such that if $g$ solves \eqref{FB} and it satisfies 
\begin{equation}\label{flatH}U(X +\eps a_0 e_n) \leq g(X) \leq U(X+\eps b_0e_n) \quad \textrm{in $B_\rho(X^*), $}\end{equation}with
$$\eps (b_0 - a_0) \leq \bar \eps \rho, $$  then \begin{equation}\label{impr}U(X +\eps a_1 e_n) \leq g(X) \leq U(X+\eps b_1e_n) \quad \textrm{in $B_{\eta \rho}(X^*)$, }\end{equation} with $$a_0 \leq a_1\leq b_1 \leq b_0, \quad (b_1-a_1) \leq (1-\eta)(b_0-a_0),$$ for a small universal constant $\eta$. \end{thm}

From this statement we get the desired corollary to be
used in the proof of our main result. Precisely, if $g$ satisfies
\eqref{flatH} with say $\rho=1/2$ , then we can apply Harnack inequality
repeatedly and obtain $$\label{osc2} U(X +\eps a_m e_n) \leq g(X) \leq U(X+\eps b_me_n) \quad \textrm{in $B_{\frac 1 2 \eta^{m}}(X^*)$, }\ $$with
\begin{equation}\label{osctilde}b_m-a_m \leq (b_0-a_0)(1-\eta)^m, \end{equation} for all $m$'s such that 
\begin{equation}\label{m}2\eps (1-\eta)^m \eta^{-m}(b_0-a_0) \leq 
\bar \eps.\end{equation}This implies that for all such $m$'s, the function $\tilde g_\eps$ defined in  Subsection 2.3 satisfies
\begin{equation}\label{oscg}
a_m \leq \tilde g_\eps(X) \leq b_m, \quad \textrm{in $B_{\frac 1 2 \eta^{m} - \eps}(X^*) \setminus P,$}
\end{equation} with $a_m, b_m$ as in \eqref{osctilde}. 
Let $A_\eps$ be the following set
\begin{equation}\label{Aeps} A_{\eps} := \{(X, \tilde g_\eps(X))  : X \in B_{1-\eps} \setminus P\} \subset \R^{n+1} \times [a_0,b_0].\end{equation}
Since $\tilde g_\eps$ may be multivalued, we mean that given $X$ all pairs $(X, \tilde g_\eps(X))$ belong to $A_\eps$ for all possible values of $\tilde g_\eps(X).$ In view of \eqref{oscg} we then get
\begin{equation}\label{oscA} A_\eps \cap (B_{\frac 1 2 \eta^{m} - \eps}(X^*) \times [a_0,b_0]) \subset B_{\frac 1 2 \eta^{m} - \eps}(X^*) \times [a_m,b_m],\end{equation}  with $a_m, b_m$ as in \eqref{osctilde} for all $m$'s such that \eqref{m} holds.

Thus we get the following corollary.

\begin{cor} \label{corHI}If  
$$U(X - \eps e_n) \leq g(X) \leq U(X+\eps e_n) \quad \textrm{in $B_1$,}$$ with $\eps \leq \bar \eps/2$, given $m_0>0$ such that $$2\eps (1-\eta)^{m_0} \eta^{-m_0} \leq \bar\eps,$$ then the set $A_\eps \cap (B_{1/2} \times [-1,1])$  is above the graph of a function $y = a_\eps(X)$ and it is below the graph of a function $y = b_\eps(X)$ with
$$ b_\eps - a_\eps \leq 2(1 - \eta)^{m_0-1},$$
and $a_\eps, b_\eps$ having a modulus of continuity bounded by the H\"older function $\alpha t^\beta$ for $\alpha, \beta$  depending only on $\eta$. 
\end{cor}


The proof of Harnack inequality will easily follow from the Lemma below.
 
\begin{lem}\label{babyH}
There exists $\bar \eps > 0$ such that for all  $0 < \eps \leq \bar \eps$ if $g$ is a solution to \eqref{FB}  in $B_1$ such that  
\begin{equation}g(X) \geq U(X) \quad \text{in $B_{1/2},$}\end{equation} and at $\bar X  \in B_{1/8}( \frac{1}{4} e_n)$
\begin{equation}\label{Bound} g(\bar X) \geq U(\bar X + \eps e_n), 
\end{equation} then
 \begin{equation}\label{onesideimprov} 
 g(X) \geq U(X + \tau \eps e_n)  \quad \textrm{in $B_{\delta}$},  \end{equation} for universal constants $\tau, \delta.$
 Similarly, if
 \begin{equation*} g(X) \leq U(X) \quad \text{in $B_{1/2}$,}
\end{equation*} and 
$$g(\bar X) \leq U(\bar X - \eps e_n),$$ then
 \begin{equation*} 
 g(X) \leq U(X -\tau\eps e_n) \quad \textrm{in $B_{\delta}$}.\end{equation*}
\end{lem}

The main tool in the proof of Lemma \ref{babyH}  will be the following family of radial subsolutions.
Let $R>0$ and denote by $$V_R(t,z) = U(t,z)((n-1)\frac{t}{R} + 1 ). $$ Then set
\begin{equation}v_R(X)= V_R(R- \sqrt{|x'|^2+(x_n-R)^2}, z),\end{equation}
that is we obtain the $n+1$-dimensional function $v_R$ by rotating the 2-dimensional function $V_R$ around  $(0,R,z).$

\begin{prop}\label{sub} If $R$ is large enough, the function $v_R(X)$ is a comparison subsolution to \eqref{FB} in $B_2$ which is strictly monotone increasing in the $e_n$-direction in $B_2^+(v_R)$. Moreover, there exists a function $\tilde v_R$ such that 
\begin{equation*}
U(X) = v_R(X - \tilde v_R(X)e_n) \quad \text{in $B_1\setminus P,$}
\end{equation*}
and
\begin{equation*}
|\tilde v_R(X) - \gamma_R(X)| \leq \frac{C}{R^2} |X|^2, \quad \gamma_R(X)=- \frac{|x'|^2}{2R} + 2(n-1)\frac{x_n r}{R},
\end{equation*} with $r= \sqrt{x_n^2+z^2}$ and $C$ universal.
\end{prop}

The proof of Proposition \ref{sub} follows from long and tedious computations and we postpone it till the Appendix. Using the estimate for $\tilde v_R$ in Proposition \ref{sub} and Lemma \ref{elem}, we also obtain the following Corollary which will be crucial for the proof of Lemma \ref{babyH}. Its proof is again presented in the Appendix.

\begin{cor}\label{corest}There exist $\delta, c_0,C_0, C_1$ universal constants, such that 
\begin{equation}
\label{2} v_R(X+ \frac{c_0}{R}e_n) \leq (1+\frac{C_0}{R}) U(X), \quad \text{in $\overline{B}_1 \setminus B_{1/4}$},\end{equation} with strict inequality on $F(v_R(X+ \frac{c_0}{R} e_n)) \cap  \overline{B}_1 \setminus B_{1/4},$
\begin{align}
\label{4}& v_R(X + \frac{c_0}{R}e_n) \geq U(X + \frac{c_0}{2R} e_n), \quad \text{in $B_{\delta},$}\\
\label{3}& v_R(X - \frac{C_1}{R} e_n) \leq U(X), \quad \text{in $\overline{B}_1.$}
\end{align}
\end{cor}

\

We are now ready to present the proof of Lemma \ref{babyH}.

\

\textit{Proof of Lemma \ref{babyH}.} We prove the first statement.
In view of \eqref{Bound}
$$g(\bar X) - U(\bar X) \geq U(\bar X+\eps e_n) - U(\bar X) = \p_tU(\bar X+ \lambda e_n) \eps \geq c\eps, \quad \lambda \in (0,\eps).$$ As in Lemma \ref{basic} we then get 
\begin{equation} \label{gU2}g(X) \geq (1+ c'\eps) U(X) \quad \text{in $\overline{B}_{1/4}.$}\end{equation} Now let 
$$R= \frac{C_0}{c'\eps},$$
where from now on the $C_i, c_i$ are the constants in Corollary  \ref{corest}. Then, for $\eps$ small enough $v_R$ is a subsolution to \eqref{FB} in $B_2$ which is monotone increasing in the $e_n$- direction and it also satisfies \eqref{2}--\eqref{3}. 
We now wish to apply the Comparison Principle as stated in Corollary \ref{compmon}. Let
$$v_R^t(X) = v_R(X+t e_n), \quad X \in B_1,$$ then according to \eqref{3}, $$v_R^{t_0} \leq U \leq g \quad \text{in $\overline{B}_{1/4}$, with $t_0=-C_1/R.$}$$ Moreover, from \eqref{2} and \eqref{gU2} we get that for our choice of $R$,
$$v_R^{t_1} \leq (1+c'\eps) U \leq g \quad \text{on $\p B_{1/4}$, with $t_1= c_0/R,$}$$ with strict inequality on $F(v_R^{t_1}) \cap \p B_{1/4}.$ 
In particular 
$$g > 0 \quad \text{on $\mathcal{F}(v_R^{t_1})$ in $\p B_{1/4}$}.$$ 
Thus we can apply  Corollary \ref{compmon} in the ball $B_{1/4}$ to obtain 
$$v_R^{t_1} \leq g \quad \text{in $B_{1/4}$.}$$
From \eqref{4} we have that 
$$U(X+\frac{c_1}{R}e_n) \leq v_R^{t_1}(X) \leq g (X) \quad \text{on $B_{\delta}$}$$ which is the desired claim \eqref{onesideimprov} with $\tau= \frac{c_1 c'}{C_0}$. 
\qed

\

We now present the proof of the Harnack Inequality.

\

\textit{Proof of Theorem \ref{mainH}.}  Without loss of generality, we can assume $a_0=-1, b_0=1$. Also, in view of Remark \ref{rescale} we can take $\rho=1$ (thus $2\eps \leq \bar \eps$).

We distinguish several cases. In what follows $\bar \eps$ and $\delta$ denote  the universal constants in Lemma \ref{babyH}.

\

{\bf Case 1.} If $$d(X^*, \{x_n \geq \eps, z=0\}) > \delta/16,$$ then $U(X-\eps e_n) >0$ in $B_{\delta/16}( X^*) \subset B_1(X^*).$ Thus the functions $U(X-\eps e_n), U(X+\eps e_n)$ and $g(X)$ are positive and harmonic in $B_{\delta/16}(X^*)$. Assume that (the other case is treated similarly)
$$g(X^*)\geq U(X^*). $$ Then,
$$g(X^*)\geq U( X^*) = U(X^* - \eps e_n) +  U_n(X^* - \lambda \eps e_n)\eps, \quad \lambda \in (0,1).$$

Since $U_n$ is positive and harmonic in $B_{\delta/16}(X^*)$ and for $\bar \eps < \delta/16$
$$X^* - \lambda \eps e_n \in B_{\delta/32}(X^*),$$ we can apply Harnack inequality to 
conclude that 
$$g(X^*) \geq U(X^*-\eps e_n) + c U_n(X^*)\eps, $$
for $c$ small universal.

Then again by Harnack inequality in $B_{\delta/16}(X^*)$ for  $g(X) - U(X - \eps e_n) \geq 0$ we get that for $c'$ universal

$$g(X) \geq U(X-  \eps e_n) + c'  U_n(X^* )\eps, \quad \text{in $B_{\delta/32}(X^*)$}.$$

By a similar argument, for $\bar \eps < \delta/32$

$$U(X-(1-\eta)\eps e_n) - U(X-\eps e_n) \leq C U_n(X^*) \eta \eps, \quad \text{in $B_{\delta/64}(X^*),$}$$ with $C$ universal.

Thus, combining these two last inequalities we obtain that for $\eta = \min\{c'/C, \delta/64\}$

$$g(X) \geq U(X-(1-\eta)\eps e_n), \quad \text{in $B_\eta (X^*),$}$$
as desired.

\

{\bf Case 2.} If $$d(X^*, \{x_n=-\eps, z=0\}) \leq \delta/2,$$ we wish to apply Lemma \ref{babyH}. 
Then (for $\bar \eps < \delta/4$)$$g(X) \geq U(X- \eps e_n) \quad \text{in $B_{1/2}(\eps e_n) \subset B_1(X^*).$}$$ 
Let $\bar X = \frac{1}{4}e_n$  and assume that  (the other case follows similarly)
$$g(\bar X) \geq U(\bar X).$$
Since  (for $\bar \eps$ small) $$\bar X \in B_{1/8}((\frac{1}{4}+\eps)e_n),$$ we can apply Lemma \ref{babyH} and conclude that
$$g(X) \geq U(X- (1-\eta)\eps e_n), \quad \text{in $B_{\delta}(\eps e_n).$}$$

Thus the desired improvement holds by choosing $\eta< \delta/4 $. Indeed for such $\eta$ and $\bar \eps < \delta/4$ we have that $d(X^*, \{x_n=\eps, z=0\}) \leq 3\delta/4$ and hence 
$$B_\eta(X^*) \subset B_{\delta}(\eps e_n).$$

\

{\bf Case 3.}  If $$d(X^*, \{x_n=-\eps, z=0\}) > \delta/2 \quad \text{and} \quad d(X^*, \{x_n \geq \eps, z=0\}) \leq \delta/16,$$ then 
the functions $U(X- \eps e_n), U(X+ \eps e_n)$ and $g(X)$ are positive and harmonic in the half ball $B_{\delta /4}(\tilde X) \cap \{z>0\}$ for some $\tilde X \in \{x_n \leq -\eps, z=0\}$ and they all vanish continuously on $B_{\delta/4}(\tilde X) \cap \{z=0\}.$ Thus we can repeat a similar argument as in Case 1, by using the boundary Harnack inequality. 
Precisely, let $\bar X=  \tilde X + \frac{\delta}{6}e_{n+1}$ and assume that 
(the other case is treated similarly)
$$g(\bar X)\geq U(\bar X).$$  Then, 
$$g(\bar X)\geq U( \bar X) = U(\bar X -  \eps e_n) +  U_n(\bar X -\lambda \eps e_n)\eps, \quad \lambda \in (0,1).$$

By Harnack inequality for $U_n$ in the ball $B_{2\eps}(\bar X) \subset B_{\delta/4}(\tilde X) \cap \{z>0\}$ (with $\bar \eps < \delta/12$) we conclude that
\begin{equation} \label{comput}g(\bar X) \geq U(\bar X -\eps e_n) + c  U_n(\bar X)\eps,\end{equation}
for $c$ small universal.

Then by Boundary Harnack inequality in $B_{\delta/4}(\tilde X) \cap \{z>0\}$, for the functions $g(X) -U(X-\eps e_n)$ and $U_n(X)$ we get that for $c'$ universal
\begin{equation}\label{gontop}g(X) \geq U(X- \eps e_n) + c'U_n(X) \eps, \quad \text{in $B_{\delta/8}(\tilde X) \cap \{z \geq 0\}$}.\end{equation}

Thus to obtain the desired claim it is enough to choose $\eta$ small such that for $X \in B_{\delta/8}(\tilde X) \cap \{z\geq 0\}$
$$U(X- \eps e_n) + c' U_n(X ) \eps \geq U(X- (1-\eta) \eps e_n).$$ 
By a similar argument as above 
$$U(\bar X- (1-\eta) \eps e_n) - U(\bar X- \eps e_n) \leq C U_n(\bar X) \eta \eps,$$ and hence by boundary Harnack inequality,
$$U( X- (1-\eta) \eps e_n) - U( X- \eps e_n) \leq C' U_n(X) \eta \eps, \quad \text{in $B_{\delta/8}(\tilde X) \cap \{z \geq 0\}$}.$$  

Combining this inequality with \eqref{gontop} we obtain that for $\eta=c'/C'$
$$g(X) \geq U( X- (1-\eta) \eps e_n), \quad \text{in $B_{\delta/8}(\tilde X) \cap \{z \geq 0\}$}.$$

Since all the functions involved are even with respect to $\{z=0\}$  and for $\eta < \delta/16$ $$B_\eta(X^*) \subset B_{\delta/8}(\tilde X),$$ our proof is complete.
\qed

\section{Improvement of flatness.}

In this Section we state the improvement of flatness property for solutions to \eqref{FB} and we provide its proof. Our main Theorem \ref{mainT} follows from Theorem \ref{iflat} and Lemma \ref{hyp}.

\begin{thm}[Improvement of flatness]\label{iflat}There exists $\bar \eps > 0$ and $\rho>0$ universal constants such that for all  $0 < \eps \leq \bar \eps$ if $g$ solves \eqref{FB}  with $0 \in F(g)$ and it satisfies 
\begin{equation}\label{flatimp}U(X - \eps e_n) \leq g(X) \leq U(X+\eps e_n) \quad \textrm{in $B_1$,}\end{equation} then 
\begin{equation}\label{flatimp2}U(x \cdot \nu  - \frac \eps 2 \rho, z) \leq g(X) \leq U(x\cdot \nu+\frac \eps 2 \rho, z) \quad \textrm{in $B_\rho$},\end{equation} for some direction $\nu \in \R^n, |\nu|=1.$ 
\end{thm}
 
The proof of Theorem \ref{iflat} will easily follow from the next four lemmas. 

\begin{lem} \label{seclem}Let $g$ be a solution to \eqref{FB} with $0\in F(g)$ and satisfying \eqref{flatimp}. Assume that the corresponding  $\tilde g_\eps$ satisfies 
\begin{equation}\label{bound}  a_0 \cdot x' - \frac{1}{4} \rho\leq  \tilde g_\eps(X) \leq a_0 \cdot x' + \frac{1}{4} \rho \quad \text{in $B_{2\rho} \setminus P,$}\end{equation} for some $a_0 \in \R^{n-1}$. Then if $\eps \leq \bar \eps (a_0,\rho)$  $g$ satisfies \eqref{flatimp2} in $B_\rho$.
\end{lem} 
\begin{proof} We prove that the lower bound holds (the upper bound can be proved similarly.)

Let,
$$\nu = (\nu', \nu_n) := \frac{(0,1) + \eps(a_0,0)}{\sqrt{1+\eps^2a_0^2}},$$
and call
$$u(X)= U(x \cdot \nu  - \frac \eps 2 \rho, z).$$ Notice that since $\nu_n>0$, $u$ is strictly monotone increasing in the $e_n$-direction say in $B_{2\rho}^+(u).$
Also, we can easily compute $\tilde u_\eps$ by its definition. Indeed, the identity
$$u(X -\eps \tilde u_\eps(X)e_n) = U(X), \quad  X\in \R^{n+1} \setminus P$$ reads as
$$U(x'\cdot \nu' + x_n \nu_n - \eps \tilde u_\eps(X)\nu_n - \frac \eps 2 \rho, z) = U(x_n,z),$$ and hence
\begin{equation}\label{formula}\tilde u_\eps (X) = \frac{x'\cdot \nu' + (\nu_n-1) x_n}{\eps \nu_n} - \frac{\rho}{2\nu_n}.\end{equation}

Thus, according to Lemma \ref{elem} it suffices to show that 
$$\tilde u_\eps \leq \tilde g_\eps \quad \text{in $B_{\rho+\eps} \setminus P,$}$$ and hence in view of \eqref{bound} we must show that
$$\tilde u_\eps(X) \leq  a_0 \cdot x' - \frac{1}{4} \rho\, \quad \text{in $B_{\rho+\eps} \setminus P.$}$$

From the choice  of $\nu$ we see that
$$\frac{\nu'}{\eps \nu_n} = a_0,$$ and
$$\frac{|\nu_n -1|}{\eps \nu_n} = \frac{1-\nu_n}{\eps \nu_n} \leq \eps a_0^2.$$
Thus, in view of the formula \eqref{formula} the desired inequality reduces to
$$x' \cdot a_0 + 2\rho \eps a_0^2  - \frac{\rho}{2} \leq x' \cdot a_0 -\frac \rho 4,$$
which is trivially satisfied for $\eps$ small enough (depending on $a_0, \rho$.)
\end{proof}

The next lemma follows immediately from the Corollary \ref{corHI} to Harnack inequality. 

\begin{lem}\label{ginfty} Let $\eps_k \rightarrow 0$ and let $g_k$ be a sequence of solutions to \eqref{FB} with $0 \in F(g_k)$ satisfying \begin{equation}\label{flatimp_k}U(X - \eps_k e_n) \leq g_k(X) \leq U(X+\eps_k e_n) \quad \textrm{in $B_1$.}\end{equation}
Denote by  $\tilde g_k$ the  $\eps_k$-domain variation of $g_k$.
Then the sequence of sets 
$$A_k := \{(X, \tilde g_k (X)) : X \in B_{1-\eps_k} \setminus P\},$$ 
has a subsequence that converge uniformly (in Hausdorff distance) in $B_{1/2} \setminus P$ to the graph $$A_\infty := \{(X,\tilde g_\infty(X)) : X \in B_{1/2} \setminus P\},$$ where $\tilde g_\infty$ is  a Holder continuous function.
\end{lem} 

From here on $\tilde g_\infty$ will denote the function from Lemma \ref{ginfty}.
\begin{lem} \label{limitsol}The function $\tilde g_\infty$ satisfies the linearized problem \eqref{linear} in $B_{1/2}$.
\end{lem} 
\begin{proof}
We start by showing that $U_n \tilde g_\infty$ is harmonic in $B_{1/2} \setminus P. $ 

Let $\tilde \varphi$ be a $C^2$ function which touches $\tilde g_\infty$ strictly by below at $X_0 \in B_{1/2} \setminus P.$ We need to show that 
\begin{equation}\label{des} \Delta (U_n \tilde \varphi)(X_0) \leq 0.
\end{equation}
Since by the previous lemma, the sequence $A_k$ converges uniformly to $A_\infty$ in $B_{1/2} \setminus P$ we conclude that there exist a sequence of constants $c_k \rightarrow 0$ and a sequence of points $X_k \in B_{1/2} \setminus P$, $X_k \rightarrow X_0$ such that $\tilde \varphi_k := \tilde \varphi + c_k$ touches $\tilde g_k$ by below at $X_k $ for all $k$ large enough. 

Define the function $\varphi_k$ by the following identity
\begin{equation}\label{varphi}\varphi_k (X- \eps_k \tilde{\varphi_k}(X) e_n) = U(X). \end{equation}

Then according to \eqref{gtildeg} $\varphi_k$ touches $g_k$ by below at $Y_k = X_k - \eps_k \tilde \varphi_k(X_k)e_n \in B_1^+(g_k),$ for $k$ large enough. Thus, since $g_k$ satisfies \eqref{FB} in $B_1$ it follows that
\begin{equation}\label{sign}
\Delta \varphi_k(Y_k) \leq 0.
\end{equation}

Let us compute $\Delta \varphi_k (Y_k).$ Since $\tilde \varphi$ is smooth, for any $Y$ in a neighborhood of $Y_k$ we can find a unique $X=X(Y)$ such that
\begin{equation}\label{Y}Y= X- \eps_k \tilde \varphi_k (X) e_n.\end{equation}
Thus \eqref{varphi} reads 
$$\varphi_k (Y) = U(X(Y)),$$with $Y_i=X_i$ if $i\neq n$ and $$\frac{\p X_j}{\p Y_i} = \delta_{ij}, \quad \textrm{when $j\neq n$}.$$ 
Using these identities we can compute that

\begin{equation}\label{Deltavarphi}\Delta \varphi_k (Y) = U_n(X) \Delta X_n(Y) + \sum_{j \neq n} (U_{jj}(X) + 2 U_{jn}(X) \frac{\p X_n}{\p Y_j} )+ U_{nn}(X) |\nabla X_n|^2(Y).\end{equation}

From \eqref{Y} we have that
$$D_X Y = I - \eps_k D_X (\tilde \varphi_k e_n).$$
Thus, since $\tilde \varphi_k = \tilde \varphi +c_k$
$$D_Y X = I + \eps_k  D_X (\tilde \varphi e_n)+ O(\eps_k^2), $$
with a constant depending only on the $C^2$-norm of $\tilde \varphi.$

It follows that 

\begin{equation}\label{pxn} \frac{\p X_n}{\p Y_j} = \delta_{jn} + \eps_k \p_j\tilde \varphi (X) + O(\eps_k^2). \end{equation}

Hence

\begin{equation}\label{nablaxn}|\nabla X_n|^2 (Y)= 1+ 2\eps_k \p_n \tilde \varphi (X)+ O(\eps_k^2), \end{equation}

and also, 

$$\frac{\p^2 X_n}{\p Y_j^2} = \eps_k \sum_{i} \p_{ji} \tilde \varphi \frac{\p X_i}{\p Y_j}+ O(\eps_k^2)= \eps_k \sum_{i \neq n} \p_{ji} \tilde \varphi \delta_{ij} + \eps_k \p_{jn}\tilde \varphi \frac{\p X_n}{\p Y_j} + O(\eps_k^2),$$
from which we obtain that
\begin{equation}\label{deltaxn}\Delta X_n = \eps_k \Delta \tilde \varphi + O(\eps_k^2).\end{equation}

Combining \eqref{Deltavarphi} with \eqref{nablaxn} and \eqref{deltaxn} we get that

$$\Delta \varphi_k (Y) =\Delta U(X) +\eps_k U_n \Delta \tilde \varphi  +2\eps_k \nabla \tilde \varphi \cdot \nabla U_n + O(\eps_k^2)(U_n(X) + U_{nn}(X)).$$

Using \eqref{sign} together with the fact that $U$ is harmonic at $X_k$ we conclude that 

$$0 \geq  \Delta (U_n \tilde \varphi) (X_k) + O(\eps_k^2)(U_n(X_k) + U_{nn}(X_k)).$$

The desired inequality \eqref{des} follows by letting $k \rightarrow +\infty.$

Next we need to show that $$|\nabla_r \tilde g_\infty |(X_0)= 0, \quad X_0=(x'_0,0,0) \in B_{1/2} \cap L,$$ in the viscosity sense of Definition \ref{linearsol}. 

Assume by contradiction that there exists a function $\phi$ which  touches $\tilde g_\infty$ by below at $X_0=(x'_0,0,0) \in B_{1/2} \cap L$ and such that $$\phi(X) = \phi(X_0) + a(X_0)\cdot (x' - x'_0) + b(X_0) r + O(|x'-x'_0|^2 + r^{3/2}), $$ with $$b(X_0) >0.$$   

Then we can find constants $\alpha,  \delta, \bar r$  and a point $Y'=(y'_0,0,0) \in B_2$ depending on $\phi$ such that the polynomial

$$q(X)=\phi(X_0) - \frac{\alpha}{2}|x'-y'_0|^2 + 2\alpha (n-1)x_n r$$
 touches $\phi$ by below at $X_0$  in a tubular neighborhood  $N_{\bar r}= \{|x'-x'_0|\leq \bar r, r \leq \bar r\}$ of $X_0,$ with 
 
 $$\phi- q \geq \delta>0, \quad \text{on $N_{\bar r} \setminus N_{\bar r/2}$.}$$ This implies that
\begin{equation}\label{second}
\tilde g_\infty - q \geq \delta>0, \quad \text{on $N_{\bar r} \setminus N_{\bar r/2}$,}\end{equation}
and 
\begin{equation}\label{third}
\tilde g_\infty (X_0)- q(X_0) =0.\end{equation}
In particular,\begin{equation}\label{thirdprime}
|\tilde g_\infty (X_k)- q(X_k)| \rightarrow 0, \quad X_k \in N_{\bar r} \setminus P, X_k \rightarrow X_0. \end{equation}

Now, let us choose $R_k=1/(\alpha \eps_k)$ and let us define
$$w_{k}(X) = v_{R_k}(X-Y'+\eps_k \phi(X_0) e_n), \quad Y'=(y'_0,0,0),$$
with $v_R $ the function defined in Proposition \ref{sub}. Then the $\eps_k$-domain variation of $w_k$, which we call $\tilde w_k$,  can be easily computed from the definition
$$w_k(X - \eps_k \tilde w_k(X)e_n)=U(X).$$
Indeed, since $U$ is constant in the $x'$-direction, this identity is equivalent to
$$v_{R_k}(X-Y'+\eps_k \phi(X_0) e_n - \eps_k \tilde w_k(X)e_n) = U(X-Y'),$$
which in view of Proposition \ref{sub} gives us
$$\tilde v_{R_k}(X-Y') = \eps_k(\tilde w_k(X) -\phi(X_0)).$$
From the choice of $R_k$, the formula for $q$ and \eqref{estvr}, we then conclude that
$$\tilde w_k (X) = q(X) + \alpha^2 \eps_k O(|X-Y'|^2),$$ and hence
\begin{equation}\label{first} |\tilde w_k - q| \leq C\eps_k \quad \text{in $N_{\bar r} \setminus P.$}\end{equation}
Thus, from the uniform convergence of $A_k$ to $A_\infty$ and \eqref{second}-\eqref{first} we get that for all $k$ large enough
\begin{equation}\label{fcont}
\tilde g_k - \tilde w_k \geq \frac \delta 2 \quad \text{in $(N_{\bar r} \setminus N_{\bar r/2}) \setminus P.$} 
\end{equation}
Similarly, from the uniform convergence of $A_k$ to $A_\infty$ and \eqref{first}-\eqref{thirdprime} we get that for $k$ large
\begin{equation}\label{scont}
\tilde g_k(X_k) - \tilde w_k(X_k) \leq \frac \delta 4,  \quad \text{ for some sequence $X_k \in N_{\bar r} \setminus P, X_k \rightarrow X_0.$}\end{equation} 

On the other hand, it follows from Lemma \ref{linearcomp} and \eqref{fcont} that
\begin{equation*}
\tilde g_k - \tilde w_k \geq \frac \delta 2 \quad \text{in $N_{\bar r} \setminus P,$} 
\end{equation*} which contradicts \eqref{scont}.
\end{proof}

The lemmas above allow us to reduce the regularity question for our free boundary problem \eqref{FB} to the regularity of the linear problem \eqref{linear}. We will analyze such question in the next section and we will consequently obtain the following lemma, which we use here to conclude the proof of our improvement of flatness Theorem.

\begin{lem}\label{reg} There exists a universal constant $\rho>0$ such that $\tilde g_\infty$ satisfies \begin{equation}\label{boundlin}  a_0 \cdot x' - \frac{1}{8} \rho\leq  \tilde g_\infty(X) \leq a_0 \cdot x' + \frac{1}{8} \rho \quad \text{in $B_{2\rho},$} \end{equation}  for a vector $a_0 \in \R^{n-1}.$
\end{lem}

 We are now ready to prove our main Theorem, by combining all of the lemmas above.
 
 \
 
 \textit{Proof of Theorem \ref{iflat}.} 
Let $\rho$ be the universal constant from Lemma \ref{reg} and assume by contradiction that we
can find a sequence $\eps_k \rightarrow 0$ and a sequence $g_k$ of
solutions to \eqref{FB} in $B_1$ such that $g_k$ satisfies \eqref{flatimp},
i.e.
 \begin{equation}\label{flatimp_k2}U(X - \eps_k e_n) \leq g_k(X) \leq U(X+\eps_k e_n) \quad \textrm{in $B_1$,}\end{equation}
 but it does not satisfy the conclusion of the Theorem.
 
Denote by  $\tilde g_k$ the  $\eps_k$-domain variation of $g_k$.
Then by Lemma \ref{ginfty} the sequence of sets
$$A_k := \{(X, \tilde g_k (X)) : X \in B_{1-\eps_k} \setminus P\},$$ 
converges uniformly (up to extracting a subsequence) in $B_{1/2} \setminus P$ to the graph $$A_\infty := \{(X,\tilde g_\infty(X)) : X \in B_{1/2} \setminus P\},$$ where $\tilde g_\infty$ is  a Holder continuous function in $B_{1/2}$. By Lemma \ref{limitsol}, the function $\tilde g_\infty$ solves the linearized problem \eqref{linear} and hence by Lemma \ref{reg} $\tilde g_\infty$ satisfies 
\begin{equation}\label{bound2}  a_0 \cdot x' - \frac{1}{8} \rho\leq  \tilde g_\infty(X) \leq a_0 \cdot x' + \frac{1}{8} \rho \quad \text{in $B_{2\rho},$}\end{equation} 
 with $a_0 \in \R^{n-1}$.

From the uniform convergence of $A_k$ to $A_\infty$, we get that for all $k$ large enough
\begin{equation}\label{boundnew}  a_0 \cdot x' - \frac{1}{4} \rho\leq  \tilde g_k(X) \leq a_0 \cdot x' + \frac{1}{4} \rho \quad \text{in $B_{2\rho} \setminus P,$}\end{equation}  and hence from Lemma \ref{seclem},  the $g_k$ satisfy the conclusion of our Theorem (for $k$ large). We have thus reached a contradiction.
\qed

\section{The regularity of the linearized problem.}

The purpose of this section is to prove an improvement of flatness result for viscosity solutions to the linearized problem associated to \eqref{FB}, that is
\begin{equation}\label{linear2}\begin{cases} \Delta (U_n w) = 0, \quad \text{in $B_1 \setminus P,$}\\ |\nabla_r w|=0, \quad \text{on $B_1\cap L$,}\end{cases}\end{equation}
where we recall that
for  $X_0=(x'_0,0,0) \in B_1 \cap L,$ we set
$$|\nabla_r w| (X_0) := \di\lim_{(x_n,z)\rightarrow (0,0)} \frac{w(x'_0,x_n, z) - w (x'_0,0,0)}{r}, \quad   r^2=x_n^2+z^2 .$$

We remark that if we restrict this linear problem to the class of functions $w(X)=\tilde w(x',r)$ that depend only on $(x',r)$ then the problem reduces to the classical Neumann problem 

$$\begin{cases} \Delta \tilde w = 0, \quad \text{in $B_1^+,$}\\ \tilde w_r=0, \quad \text{on $\{r=0\}$.}\end{cases}$$

The following is our main theorem.

\begin{thm}\label{class} Given a boundary data $\bar h \in C(\p B_1), |\bar h| \leq 1,$ which is even with respect to $\{z=0\}$, there exists a unique classical solution $h$ to \eqref{linear2} such that $h \in C(\overline{B}_1)$, $h = \bar h$ on $\p B_1$, $h$ is even with respect to $\{z=0\}$ and it satisfies
\begin{equation}\label{mainh}|h(X) - h(X_0) -  a' \cdot (x'- x'_0)| \leq C (|x'-x'_0|^2 + r^{3/2}), \quad X_0 \in B_{1/2} \cap L,\end{equation} for a universal constants $C$ and a vector $a' \in \R^{n-1}$ depending on $X_0.$
\end{thm}

As a corollary of the theorem above we obtain the desired regularity result, as stated also in Section 3.

\begin{thm}[Improvement of flatness]\label{lineimpflat} There exists a universal constant $C$ such that if $w$ is a viscosity solution to \eqref{linear2} in $B_1$ with $$-1 \leq w(X) \leq 1\quad \text{in $B_1,$}$$ then  \begin{equation}\label{boundlin2}  a_0 \cdot x' -C|X|^{3/2} \leq w(X) - w(0)\leq a_0 \cdot x' + C |X|^{3/2},\end{equation}for some vector $a_0\in \R^{n-1}$.\end{thm}
\begin{proof}
Let $h$ be the unique solution to \eqref{linear2} in $B_{1/2}$ with boundary data $w$. We will prove that $w=h$ in $B_{1/2}$ and hence it satisfies  the desired estimate in view of \eqref{mainh}. Denote by $$\bar h_\eps := h - \eps + \eps^2 r.$$ Then, for $\eps$ small $$\bar h_\eps < w \quad \text{on $\p B_{1/2}$}.$$ We wish to prove that
\begin{equation}\label{Ub} \bar h_\eps \leq w \quad \text{in $B_{1/2}$.}\end{equation}
Now, notice that  $\bar h_\eps$ (and all its translations) is a classical strict subsolution to \eqref{linear2} that is \begin{equation}\label{linear3}\begin{cases} \Delta (U_n \bar h_\eps) = 0, \quad \text{in $B_{1/2} \setminus P,$}\\ |\nabla_r \bar h_\eps|>0, \quad \text{on $B_{1/2}\cap L$.}\end{cases}\end{equation}
Since $w$ is bounded, for $t$ large enough $\bar h_\eps - t$ lies strictly below $w$. We let $t \rightarrow 0$ and show that the first contact point cannot occur for $t \geq 0$. Indeed since $\bar h_\eps -t$ is a strict subsolution which is strictly below $w$ on $\p B_{1/2}$ then no touching can occur either in $B_{1/2} \setminus P$ or on $B_{1/2} \cap L.$ We only need to check that no touching occurs on $ P \setminus L.$ This follows from Lemma \ref{rk}.

Thus \eqref{Ub} holds. 
Passing to the limit as $\eps \rightarrow 0$ we get that $$h \leq w \quad \text{in $B_{1/2}$.}$$ Similarly we also obtain that $$h \geq w \quad \text{in $B_{1/2},$}$$ and the desired equality holds. 
\end{proof}

The existence of the classical solution of Theorem \ref{class} will be achieved via a variational approach in the appropriate weighted Sobolev space. 

We say that $h \in H^1(U_n^2 dX, B_1)$ is a minimizer to the energy functional 
$$J(h) := \int_{B_1} U_n^2 |\nabla h|^2 dX,$$
if
$$J(h) \leq J(h+\phi), \quad \forall \phi \in C_0^\infty(B_1).$$
Since $J$ is strictly convex this is equivalent to
$$\di\lim_{\eps \rightarrow 0} \frac{J(h)-J(h+\eps \phi)}{\eps} =0,  \quad \forall \phi \in C_0^\infty(B_1),$$ which is satisfied if and only if
$$\int_{B_1} U_n^2 \nabla h \cdot \nabla \phi \;dX = 0, \quad  \forall \phi \in C_0^\infty(B_1). $$

As remarked above, if we restrict to the space of functions $h$ which are axis-symmetric with respect to $L$ then the energy above reduces to the Dirichlet energy.

\

We start with a few standard facts about minimizers of $J$. First, $h$ solves the equation
$$\textrm{div}(U_n^2 \nabla h) = 0 \quad \text{in $B_1,$}$$ which is uniformly elliptic in any compact subset of $B_1 \setminus P$ where $U_n$ is bounded. 
In particular $h \in C^{\infty}(B_1 \setminus P),$ and we easily obtain the following lemma.

\begin{lem}\label{Unhisharm} Let $h$ be a minimizer to $J$ in $B_1$, then $$\Delta (U_n h) = 0 \quad \text{in $B_1 \setminus P.$}$$
\end{lem}
\begin{proof} Since $h$ is smooth in $B_1 \setminus P$, from
 $$\textrm{div}(U_n^2 \nabla h) = 0 \quad \text{in $B_1,$}$$ we obtain that
 $$U_n^2 \Delta h + 2\sum_{i=1}^{n+1} U_n U_{ni}h_i = 0 \quad \text{in $B_1 \setminus P.$}$$
Since $U_n > 0$ and $\Delta U=0$ in $B_1 \setminus P$ the identity above is equivalent to
$$\Delta(U_n h) = U_n \Delta h + 2 \nabla U_n \cdot \nabla h=0 \quad \text{in $B_1 \setminus P,$}$$ as desired.
\end{proof}

The next lemma contains a characterization of minimizer, which we will be useful later in this section.

\begin{lem}\label{char}Let $h \in C(B_1)$ be a solution to \begin{equation}\Delta (U_n h) = 0 \quad \text{in $B_1 \setminus P,$}\end{equation}
and assume that 
$$\di\lim_{r \rightarrow 0} h_r(x',x_n,z) = b(x'),$$
with $b(x')$  a continuous function. Then $h$ is a minimizer to $J$ in $B_1$ if and only if $b \equiv 0.$\end{lem}
\begin{proof}
By integration by parts and the computation in Lemma \ref{Unhisharm} the identity 
$$\int_{B_1} U_n^2 \nabla h \cdot \nabla \phi \;dX = 0, \quad  \forall \phi \in C_0^\infty(B_1),$$ is equivalent to the following two conditions
\begin{equation}\label{unvharm}\Delta (U_n h) = 0 \quad \text{in $B_1 \setminus P,$}\end{equation}
and
\begin{equation}\label{fbforv}\di\lim_{\delta \rightarrow 0} \int_{\p C_\delta \cap B_1} U_n^2 \phi \nabla h \cdot \nu d\sigma =0,\end{equation}
where $C_\delta$ is the cylinder $ \{r \leq \delta\}$
and $\nu$ the inward unit normal to $C_\delta.$ 

Here we use that 
$$\lim_{\eps \rightarrow 0}  \int_{\{|z|=\eps\} \cap (B_1 \setminus C_\delta)} U_n^2 \phi h_{\nu}  d\sigma =0.
$$

Indeed, in the set $\{|z|=\eps\} \cap (B_1 \setminus C_\delta)$ we have
$$U_n \leq C\eps, $$
and
$$|\nabla (U_nh)|, |\nabla U_n| \leq C,$$
from which it follows that
$$U_n |\nabla h| \leq C.$$

In conclusion we need to show that \eqref{fbforv} is equivalent to $b(x')=0.$

Indeed,
\begin{align*}
\int_{\p C_\delta \cap B_1} U_n^2 \phi \nabla h \cdot \nu d\sigma &=\frac 1 \delta  \int_{\p C_\delta \cap B_1} \cos^2 (\frac \theta 2) h_r \phi d\sigma\\ &= \int_{\p C_1 \cap B_1}  \cos^2 (\frac \theta 2)(h_r \phi)(X', \delta \cos \theta, \delta \sin \theta) dx' d \theta,\end{align*}
hence 
$$\lim_{\delta \rightarrow 0}\int_{\p C_\delta \cap B_1} U_n^2 \phi \nabla h \cdot \nu d\sigma = \pi \int_L b(x') \phi(x',0,0) d x' $$ and our claim clearly follows.
\end{proof}

The next lemma follows by standard arguments, hence we omit its proof.

\begin{lem}[Comparison Principle]\label{mincomp} Let $h_1, h_2$ be minimizers to $J$ in $B_1$. If $$h_1\geq h_2 \quad \text{a.e in $B_1 \setminus B_{\rho}$,}$$ then  $$h_1\geq h_2 \quad \text{a.e. in $B_1.$}$$  
\end{lem}



Finally one of the main ingredient in the proof of Theorem  \ref{class} is the following Harnack inequality. 

\begin{lem}[Harnack inequality] \label{HImin} Let $h$ be a minimizer to $J$ in $B_1$ which is even with respect to $\{z =0\}$. Then $h \in C^\alpha(B_{1/2})$ and $$[h]_{C^{\alpha}(B_{1/2})} \leq C,$$ with $C$ universal.
\end{lem}

The proof of this lemma follows the same lines as the proof of Harnack inequality (Theorem \ref{mainH}) for the free boundary problem \eqref{FB}. We briefly sketch it in what follows. 

\

{\it Sketch of the Proof of Lemma \ref{HImin}.}  The key step consists in proving the following  claim, which plays the same role as Lemma \ref{babyH} in the proof of Theorem \ref{mainH}. The remaining ingredients are the standard  Harnack inequality and Boundary Harnack inequality for harmonic functions.

\

{\it Claim:} There exist universal constants $\delta, c$ such that if $h \geq 0$ a.e. in $B_1$ and
$$h( \frac 1 4 e_n) \geq  1,$$ then
$$h \geq  c \quad \text{a.e. in $B_\delta.$}$$

As in the proof of Lemma \ref{babyH}, since minimizers satisfy the comparison principle Lemma \ref{mincomp}, 
the claim will follow if we provide the right family of comparison minimizers. This family plays the same role as the $v_R$'s in Lemma \ref{babyH} and it is obtained by translations and multiplication by constants of the following function
$$v(X) := -\frac{|x'|^2}{n-1} + 2x_n r.$$ 
We need to show that
$v$ is a  minimizer to $J$ in $B_1$. To do so we prove that $v$ satisfies Lemma \ref{char}. 

To prove that  \begin{equation*}\Delta (U_n v) = 0 \quad \text{in $B_1 \setminus P,$}\end{equation*} we use that $2r U_n = U$ and that $U, U_n$ are harmonic outside of $P$ and do not depend on $x'$. Thus

$$\Delta(U_n v) = -\Delta(\frac{|x'|^2}{n-1}U_n) + \Delta(x_n U) = -2U_n + 2U_n=0.$$

Finally the fact  that
$$\lim_{r \rightarrow 0} v_r(x',x_n,z) = 0,$$ follows immediately from the definition of $v$.
\qed

\

Since our linear problem is invariant under translations in the $x'$-direction, we see that discrete differences  of the form 
$$h(X + \tau) - h(X),$$with $\tau$ in the $x'$-direction are also minimizers. Now by standard arguments (see \cite{CC}) we obtain the following corollary. 

\begin{cor}\label{derivativeH} Let $h$ be a minimizer to $J$ in $B_1$ which is even with respect to $\{z =0\}$. Then $D_{x'}^\beta h \in C^\alpha(B_{1/2})$ and $$[D_{x'}^\beta h]_{C^{\alpha}(B_{1/2})} \leq C,$$ with $C$ depending on $\beta.$

\end{cor}

We are now ready to prove our main theorem. 

\

\textit{Proof of Theorem \ref{class}.}  We divide our proof in several steps.

\

\textbf{Step 1.} In this step, we show the existence of a classical solution to our problem, which achieves the boundary data continuously.

Assume without loss of generality  that $\bar h \in C^{\infty}(\p B_1)$. The general case when $\bar h \in C(\p B_1)$ follows by approximation and the Comparison Principle. 

We minimize $J(\cdot)$ among all functions $h$ with boundary data $\bar h,$ which are even with respect to $\{z=0\}.$ From Lemma \ref{Unhisharm} we have that $$\Delta (U_n h) = 0 \quad \text{in $B_1 \setminus P$}. $$ In Step 2-3 we will show that $h$ satisfies the estimate \eqref{mainh} and in particular $$ h(x'_0,x_n,x) - h(x'_0, 0,0)=O( r^{3/2}), \quad X_0=(x'_0,0,0) \in B_1 \cap L$$ which gives $$|\nabla_r h| = 0 \quad \text{on $B_1\cap L,$}$$ where we recall that  $$|\nabla_r h|(X_0)=\lim_{(x_n,z)\rightarrow (0,0)} \frac{h(x'_0,x_n, z) - h (x'_0,0,0)}{r}, \quad   r^2=x_n^2+z^2 .$$

Notice that since $|\bar h| \leq 1$, also $|h| \leq 1$ and $h, D_{x'}h \in C^{0,\alpha}(B_1)$ in view of  Lemma  \ref{HImin} and its Corollary.  

We now show that $h$ achieves the boundary data $\bar h$ continuously. Indeed this follows by classical elliptic theory if we restrict to $\p B_1 \setminus P.$ 

If $X_0 \in \p B_1 \cap (P \setminus L)$ then in a small neighborhood of $X_0$ intersected with $B_1 \cap \{z >0\}$ the function $U_n h$ is harmonic continuous up to the boundary and vanishes continuously on $\{z=0\}$ (since $h$ is bounded). The continuity of $h$ at $X_0$ then follows from standard boundary regularity for the harmonic function $U_n h.$

Finally, on the set $\p B_1 \cap L$ as in the case of Laplace equation, it suffices to construct at each point $X_0$ a local barrier (minimizer) for $h$ which is zero at $X_0$ and strictly negative in a neighborhood of $X_0$. 
Such barrier is given by  (see Lemma \ref{char}) $$(x'-x_0') \cdot x_0'.$$

\

\textbf{Step 2.} In this step we wish to prove that 
\begin{equation}\label{exp2} |h(x',x_n,x) - h(x', 0,0) - b(x') r| \leq C r^{3/2}, \quad (x',0,0) \in B_{1/2} \cap L, \end{equation} \begin{equation}\label{hb}|h_r(x',x_n,z) -b(x') | \leq Cr^{1/2}, \quad (x',0,0) \in B_{1/2} \cap L,\end{equation}with $C$ universal and $b(x')$ a Lipschitz function.

Indeed, $h$ solves $$\Delta(U_n h) =0 \quad \text{in $B_1 \setminus P$}. $$ Since $U_n$ is independent on $x'$ we can rewrite this equation as
\begin{equation}\label{2dreduction} \Delta_{x_n,z} (U_n  h) = - U_n \Delta_{x'} h ,\end{equation}
and according to Lemma \ref{HImin} we have that 
$$\Delta_{x'} h \in C^{\alpha}(B_{1/2}),$$ with universal bound. 
Thus, for each fixed $x'$, we need to investigate the 2-dimensional problem 
$$\Delta (U_t h) = U_t f, \quad \text{in $B_{1/2} \setminus \{t \leq 0, z=0\}$}$$
with $$f \in C^{\alpha}(B_{1/2}),$$ and $h$ bounded. Without loss of generality, for a fixed $x'$ we may assume $h(x', 0,0)=0.$

Let $H(t,z)$ be the solution to the problem
$$\Delta H = U_t f, \quad \text{in $B_{1/2} \setminus \{t \leq 0, z=0\},$}$$
such that 
$$H = U_t  h  \quad \text{on $\p B_{1/2}$}, \quad H=0 \quad \text{on $B_{1/2} \cap \{t \leq 0, z=0\}.$}$$

We wish to prove that \begin{equation}\label{equal}U_t h = H.\end{equation}

First notice that  $$ \Delta (H - U_t  h) =0 \quad  \text{in $B_{1/2} \setminus \{t \leq 0, z=0\},$}$$ and $$H = 0 \quad \text{on $\p B_{1/2} \cup (B_{1/2} \cap \{t < 0, z=0\})$}.$$
We claim that
\begin{equation}\label{limitzero}
\lim_{(t,z) \rightarrow (0,0)} \frac{H- U_t h}{U_t}=\lim_{(t,z) \rightarrow (0,0)} \frac{H}{U_t}=0.
\end{equation}
If the claims holds, then given any $\eps >0$
$$-\eps U_t \leq H - U_t h \leq \eps U_t, \quad \text{in $B_\delta$}$$
with $\delta=\delta(\eps).$
Then by the maximum principle the inequality above holds in the whole $B_{1/2}$ and by letting $\eps \rightarrow 0$ we obtain \eqref{equal}.

To prove the claim \eqref{limitzero} we show that $H$ satisfies the following
\begin{equation}\label{Hexp}
|H(t,z) - aU(t,z)| \leq C_0 r^{1/2}U(t,z), \quad r^2=t^2+z^2, a\in \R,
\end{equation} with $C_0$ universal.

To do so, we consider the holomorphic transformation 
$$\Phi: (s,y) \rightarrow (t,z) =(\frac 1 2 (s^2-y^2), sy)$$
which maps  $B_{1} \cap \{s > 0\}$ into $B_{1/2} \setminus  \{t \leq 0, z=0\}$ and call

$$\tilde H (s,y) = H(t,z), \quad \tilde f(s,y) = f(t,z).$$

Then, easy computations show that
$$\Delta \tilde H = s \tilde f \quad \text{in  $B_{1} \cap \{s > 0\}$}, \quad \tilde H=0 \quad \text{on $B_{1} \cap \{s=0\}$.}$$
Since the right-hand side is $C^\alpha$ we conclude that $\tilde H \in C^{2,\alpha}$. In particular $\tilde H_s$ satisfies 
$$|\tilde H_s (s, y) -a| \leq C_0|(s,y)|, \quad a=\tilde H_s(0,0)$$ with $C_0$ universal.
Integrating this inequality between $0$ and $s$ and using that $\tilde H= 0$ on $B_{1} \cap \{s=0\}$ we get that
$$|\tilde H(s,y) - a s| \leq C_0 s|(s,y)|. $$
In terms of $H$, this equation gives us
\begin{equation}\label{HaU}|H - a U| \leq C_0 r^{1/2} U\end{equation} as desired.

Thus \eqref{equal} and \eqref{Hexp} hold and by combining them and using that $U/U_t= 2r$ we get that 
$$|h -  2 a r| \leq 2 C_0 r^{3/2},$$
which is  the desired estimate \eqref{exp2} i.e. (recall that above we assumed $h(x',0,0)=0$)
\begin{equation*}\label{exp3} |h(x',x_n,x) - h(x', 0,0) - b(x') r| \leq 2C_0 r^{3/2}, \end{equation*}
with $$b(x') = 2\tilde H_s (x',0,0).$$

We remark that $b(x')$ is Lipschitz. Indeed, notice that the derivatives $h_i, i=1,\ldots, n-1$ still satisfy the same equation \eqref{2dreduction} as $h$, where the $C^\alpha$ norm of the right-hand side has a universal bound. Thus, we can argue as above to conclude that 
$$|\p_i \tilde H_s(x',0,0)| \leq C,$$ which together with the formula for $b(x')$ shows that $b(x')$ is a Lipschitz function.

Finally, to obtain the second of our estimate \eqref{hb} we proceed similarly as above. 

Since $U_t h= H$ one can compute easily that \begin{equation}\label{hr}U_t h_r = H_r + \frac{1}{2}\frac{H}{r}.\end{equation}
Moreover, after our holomorphic transformation
\begin{equation}\label{holoradial}2r H_r(t,z) = s \tilde H_s(s,y)+ y \tilde H_y(s,y).\end{equation}
As observed above,
$$|\tilde H_s -a(x')|\leq C |(s,y)|,$$
and similarly since  $\tilde H= 0$ on $B_{1} \cap \{s=0\}$ 
$$|\tilde H_y| \leq C s.$$
These two inequalities combined with \eqref{holoradial} give us 
\begin{equation}\label{2rHr}|2r H_r -  a(x') U| \leq Cr^{1/2}U.\end{equation}
Combining \eqref{hr} with \eqref{HaU}-\eqref{2rHr} we obtain \eqref{hb} as desired.

\

\textbf{Step 3.} In this step we show that $h$ satisfies \eqref{mainh}.

In view of Lemma \ref{char} and estimate \eqref{hb} we obtain that $b(x')=0$. Since $b(x')=0$ then \eqref{exp2} reads
$$|h(x',x_n,z) - h(x',0,0)| \leq Cr^{3/2}.$$Since $h$ is $C^{\infty}$ in the $x'$ direction, we have that for a given $X_0 \in B_{1/4} \cap L$
$$|h(x',0,0) - h(x'_0,0,0) - a' \cdot (x'-x'_0)| \leq C|x'-x'_0|^2 ,$$
which combined with the previous inequality gives us the desired bound \eqref{mainh}.

\qed

\section{Appendix}

The purpose of this Section is to provide the proof of Proposition \ref{sub} and Corollary \ref{corest} which have been used in the proof of Harnack Inequality in Section 6. For the reader's convenience we recall their statements.

Let $R>0$ and denote by $$V_R(t,z) = U(t,z)((n-1)\frac{t}{R} + 1 ). $$ Then set
\begin{equation*}v_R(X)= V_R(R- \sqrt{|x'|^2+(x_n-R)^2}, z).\end{equation*}

\textbf{Proposition 6.4.} \textit{If $R$ is large enough, the function $v_R(X)$ is a comparison subsolution to \eqref{FB} in $B_2$ which is strictly monotone increasing in the $e_n$-direction in $B_2^+(v_R)$. Moreover, there exists a function $\tilde v_R$ such that 
\begin{equation}\label{eq}
U(X) = v_R(X - \tilde v_R(X)e_n), \quad \text{in $B_1\setminus P$}
\end{equation}
and
\begin{equation}\label{estvr}
|\tilde v_R(X) - \gamma_R(X)| \leq \frac{C}{R^2} |X|^2, \quad \gamma_R(X)=- \frac{|x'|^2}{2R} + 2(n-1)\frac{x_n r}{R},
\end{equation} with $r= \sqrt{x_n^2+z^2}$ and $C$ universal.}

\begin{proof} We divide the proof of this proposition in two steps.

{\bf Step 1. } In this step we show that $v_R$ is a comparison subsolution in $B_2$ which is monotone in the $e_n$-direction. 

First we need to show that $v_R$ is subharmonic (but not harmonic) in $B_2^+(v_R)$. From the formula for $v_R$ we see immediately that ($R>2$)
$$B_2^+(v_R) = B_2 \setminus (\mathcal{B}_2  \setminus \overline{\mathcal{B}}_R(Re_n)).$$
One can easily compute that on such set,

$$\Delta v_R (X) = (( \p_{tt}+\p_{zz}))V_R)(R-\rho, z) - \frac{n-1}{\rho}\p_t V_R(R-\rho, z), $$
where for simplicity we call $$\rho :=  \sqrt{|x'|^2+(x_n-R)^2}.$$ 
Also for $(t,z)$ outside the set  $\{(t,0) : t \leq 0\} $
\begin{align*}\Delta_{t,z} V_R (t,z)&= (\p_{tt}+ \p_{zz})V_R(t,z) = \frac{2(n-1)}{R} \p_t U(t,z)+ (1+(n-1)\frac{t}{R})\Delta_{t,z} U(t,z)\\
&= \frac{2(n-1)}{R} \p_t U(t,z),\end{align*}
and 
\begin{equation}\label{pV}\p_t V_R(t,z) = (1+(n-1)\frac{t}{R})\p_t U(t,z) + \frac{n-1}{R} U(t,z).\end{equation}
Thus to show that $\Delta v_R$ is subharmonic in $B^+_2(v_R)$ we need to prove that in such set
$$\frac{2(n-1)}{R} \p_t U - \frac{n-1}{\rho}[(1+(n-1)\frac{R-\rho}{R})\p_t U + \frac{n-1}{R} U] \geq 0,$$
where $U$ and $\p_t U$ are evaluated at $(R-\rho,z).$ 

Set $ t = R -\rho$, then straightforward computations reduce the inequality above to
$$(n-1)[2(R- t) - R - (n-1)^2 t]\p_t U(t, z) - (n-1)^2U(t, z) \geq 0.$$
Using that $\p_t U( t, z)= U( t,z)/(2 r)$ with $r^2= t^2 + z^2$, this inequality becomes
$$R \geq 2 t + (n-1)^2 t + 2(n-1) r.$$
This last inequality is easily satisfied for $R$ large enough, since $t, r \leq 3.$

Now we prove that $v_R$ satisfies the free boundary condition in Definition \ref{defsub}.
First observe that $$F(v_R) = \p \mathcal{B}_R(Re_n,0) \cap \mathcal{B}_2,$$ 
and hence it is smooth. By the radial symmetry it is enough to show that the free boundary condition is satisfied at $0 \in F(v_R)$ that is
\begin{equation}\label{expa}v_R(x,z) = \alpha U(x_n,z) + o(|(x,z)|^{1/2}), \quad \text{as $(x,z) \rightarrow (0,0),$}\end{equation}
with $\alpha \geq 1.$

First notice since $U$ is Holder continuous with exponent $1/2$, it follows from the formula for $V_R$ that
$$|V_R(t,z) - V_R(t_0,z)| \leq C |t-t_0|^{1/2} \quad \text{for $|t-t_0| \leq 1.$}$$
Thus for $(x,z) \in B_s,$ $s$ small
$$|v_R(x,z) - V_R(x_n,z)| = |V_R(R-\rho, z) - V_R(x_n, z)| \leq C |R-\rho - x_n|^{1/2} \leq C s,$$
where we have used that (recall that $\rho :=  \sqrt{|x'|^2+(x_n-R)^2}$) 
\begin{equation}\label{easy}R - \rho - x_n = - \frac{|x'|^2}{R-x_n + \rho}.\end{equation}
It follows that for $(x,z) \in B_s$
\begin{align*}|v_R(x,z) - U(x_n,z)| &\leq |v_R(x,z) - V_R(x_n, z)| + |V_R(x_n,z) - U(x_n, z)|\\ & \leq Cs + |V_R(x_n,z) - U(x_n, z)|.\end{align*} Thus from the formula for $V_R$
$$|v_R(x,z) - U(x_n,z)| \leq Cs +(n-1) \frac{|x_n|}{R}U(x_n,z) \leq C' s, \quad (x,z) \in B_s$$
which gives the desired expansion \eqref{expa} with $\alpha=1.$

Now, we show that $v_R$ is strictly monotone increasing in the $e_n$-direction in $B_2^+(v_R)$. Outside of its zero plate,
$$\p _{x_n} v_R(x) = - \frac{x_n-R}{\rho} \p_t V_R(R-\rho, z).$$ Thus we only need to show that $V_R(t,z)$ is strictly monotone increasing in $t$ outside $\{(t,0) : t \leq 0\}$ . This follows immediately from \eqref{pV} and the formula for $U$.

\

{\bf Step 2.} In this step we show the existence of $\tilde v_R$ satisfying \eqref{eq} and \eqref{estvr}. Since we have a precise formula for $v_R$ in terms of $U$, this is only a matter of straightforward (though tedious) computations which we present here for completeness. 

First we show that there exists a unique $\tilde t$ such that 
(here $B_1$ is the 2-dimensional ball)
\begin{equation}\label{2D}U(t,z) = V_R(t+\tilde t, z), \quad \text{in $B_1 \setminus \{(t,0) : t \leq 0\},$}\end{equation}
and 
\begin{equation}\label{esttilde}|\tilde t + \frac{2(n-1)tr}{R}| \leq \frac{\tilde C}{R^2}r^3, \quad r^2= t^2+z^2,\end{equation} with $\tilde C$ universal.
Since $V_R$ is strictly increasing in the $t$-direction in $B_1 \setminus \{(t,0) : t \leq 0\}$ it suffices to show that

\begin{equation} \label{trapped}V_R(t- \frac{2(n-1)tr}{R} - \frac{\tilde C}{R^2}r^3) < U(t,z) <V_R(t- \frac{2(n-1)tr}{R} + \frac{\tilde C}{R^2}r^3).\end{equation}

Let us prove the lower bound. We call $$\bar t =-\frac{2(n-1)tr}{R} - \frac{\tilde C}{R^2}r^3,$$  and we use Taylor's formula to compute

\begin{equation}\label{VR}V_R(t + \bar t, z) = V_R(t,z) + \p_t V_R(t,z) \bar t + \frac{1}{2}E \bar t^2, \quad |E| \leq |\p_{tt} V_R (s,z)|,\end{equation}
with $s$ between $t$ and $t+\bar t$. We claim that 

\begin{equation}\label{claimpttvr}|\p_{tt} V_R (s,z)| \leq \frac{C'}{r^2} U(t,z). \end{equation}

Indeed one can compute that

\begin{align}\label{pttvr}\p_{tt} V_R (s,z) &= (1+(n-1)\frac{s}{R})\p_{tt}U (s,z) + \frac{2(n-1)}{R}\p_t U(s,z)\\ \nonumber &= (1+(n-1)\frac{s}{R})r^{-3/2}\p_{tt}U (\frac s r, \frac z r) + r^{-1/2}\frac{2(n-1)}{R}\p_t U(\frac s r, \frac z r)\end{align}
where we have used that $U$ is homogeneous of degree $1/2$. 

Since $s$ lies between $t$ and $t+\bar t$ we get that  $(s/r,z/r) \in B_{3/2}\setminus B_{1/2},$ thus by boundary Harnack inequality in this annulus we get

$$|\p_{tt} U(\frac s r,\frac z r)| \leq K_1 U(\frac s r,\frac z r), \quad \p_t U(\frac s r,\frac z r) \leq K_2U (\frac s r,\frac z r),$$ with $K_1, K_2$ universal. 

Combining the above inequalities with \eqref{pttvr} we obtain that

$$|\p_ {tt} V_R(s,z)| \leq \bar C r^{-3/2} U(\frac s r,\frac z r). $$

Thus the claim in \eqref{claimpttvr} will follow if we show that for some $K$ universal

$$U(\frac s r,\frac z r) \leq K U(\frac t r,\frac z r).$$

Again, this follows by the boundary Harnack inequality in the annulus $B_{3/2}\setminus B_{1/2}$ between $U(\frac t r, \frac z r)$ and its translation $U(\frac{t+\tau}{r},\frac z r)$, for $\tau$ small. Our claim is thus proved.

Thus, using \eqref{VR} together with this claim, the lower bound in \eqref{trapped} will be proved if  we show that

$$U(t,z) > V_R(t,z)  + \p_t V_R(t,z) \bar t  + \frac{C' }{2 r^2} U(t,z)\bar t^2.$$

From the definition of $V_R$ this is equivalent to showing that 

$$ (n-1)\frac{t}{R}U(t,z)  + [(1+(n-1)\frac{t}{R})\p_t U(t,z) + \frac{n-1}{R}U(t,z)] \bar t  +  \frac{C'}{2 r^2}U(t,z)\bar t^2 < 0.$$

Dividing everything by $\p_t U(t,z) = \frac{1}{2r} U(t,z)$ we get

$$ \frac{2(n-1)r t}{R}  + [(1+(n-1)\frac{t}{R}) + \frac{2r(n-1)}{R}] \bar t  +  \frac{C'}{r} \bar t^2 < 0 ,$$
and using the definition of $\bar t$ we finally need to show that

$$ (n-1)\frac{t+2r}{R} \bar t  +  \frac{C'}{r}\bar t^2 < \frac{\tilde C}{R^2}r^3.$$

Using that for $R$ large $$|\bar t| \leq K r^2/R,$$ for $K$ universal, we easily obtain that the inequality above holds for the appropriate $\tilde C$ (universal) and $R$ large.

To conclude our proof, we use \eqref{easy} to write
$$v_R(X- \tilde{v}_R e_n) = V_R(R -\rho(\tilde{v}_R), z) = V_R(x_n - \tilde{v}_R - \frac{|x'|^2}{R-x_n+\tilde{v}_R+\rho(\tilde{v}_R)},z),$$  with 
$$\rho(\eta):=  \sqrt{|x'|^2 + (x_n-\eta - R)^2}.$$

In view of \eqref{2D} if there exists $\tilde{v}_R=\tilde v_R(X)$ such that 
\begin{equation}\label{eta}-\tilde{v}_R- \frac{|x'|^2}{R-x_n+\tilde{v}_R+\rho(\tilde{v}_R)} = \tilde t,\end{equation} then 
$$v_R(X-\tilde{v}_R e_n) = U(X), \quad \text{in $B^+_1(U)$.}$$
By the strict monotonicity of $v_R$ in the $e_n$-direction in $B^+_1(v_R)$, in such set $\tilde {v}_R$ must be unique.
 
Thus our claim will follow if we show that there exists $\tilde{v}_R$ satisfying \eqref{eta} and such that
$$|\tilde{v}_R(X) - \gamma_R(X)| \leq C\frac{|X|^2}{R^2}.$$  To do so, we call
$$f(\eta) = -\eta- \frac{|x'|^2}{R-x_n+\eta+\rho(\eta)}, \quad -1 \leq \eta \leq 1,$$ and we show that
$$f(\gamma_R(X) + C \frac{|X|^2}{R^2}) \leq \tilde t \leq f(\gamma_R(X) - C \frac{|x|^2}{R^2}).$$
In view of \eqref{esttilde} we need to prove that
$$f(\gamma_R(X) + C \frac{|X|^2}{R^2}) \leq -\frac{2(n-1)x_nr}{R} - \tilde C\frac{r^3}{R^2},$$
and 
$$f(\gamma_R(X) - C \frac{|X|^2}{R^2}) \geq -\frac{2(n-1)x_nr}{R} + \tilde C\frac{r^3}{R^2}.$$

Let us prove the first inequality (the second one follows similarly.) Call $$\bar \eta = \gamma_R(X) + C \frac{|X|^2}{R^2}.$$
From the definition of $f$ and $\gamma_R$ the desired  inequality is equivalent to 
$$\frac{|x'|^2}{2R} - C\frac{|X|^2}{R^2} - \frac{|x'|^2}{R-x_n+\bar \eta+\rho(\bar \eta)} \leq - \tilde C\frac{r^3}{R^2}.$$
Clearly $-1 \leq \bar \eta \leq 1$, and one can easily verify that
$$R-x_n+\bar \eta+\rho(\bar \eta) \leq 2R+5.$$
Thus 
$$\frac{|x'|^2}{2R} - \frac{|x'|^2}{R-x_n+\bar \eta+\rho(\bar \eta)} \leq |x'|^2(\frac{1}{2R} - \frac{1}{2R+5}) \leq\frac{ |x'|^2}{R^2},$$
and the desired inequality follows if we show that
$$\frac{|x'|^2}{R^2} -  C\frac{|X|^2}{R^2} \leq - \tilde C\frac{r^3}{R^2}. $$
This inequality is trivially satisfied as long as $C - \tilde C \geq 1.$ 
\end{proof}

We now recall the statement of Corollary \ref{corest} and sketch its proof.

\
 
\noindent\textbf{Corollary 6.5.} \textit{There exist $\delta, c_0,C_0, C_1$ universal constants, such that 
\begin{equation}
\label{2a} v_R(X+ \frac{c_0}{R}e_n) \leq (1+\frac{C_0}{R}) U(X), \quad \text{in $ \overline{B}_1 \setminus B_{1/4},$}\end{equation} with strict inequality on $F(v_R(X+ \frac{c_0}{R} e_n)) \cap  \overline{B}_1 \setminus B_{1/4},$
\begin{align}
\label{4a}& v_R(X + \frac{c_0}{R}e_n) \geq U(X + \frac{c_0}{2R} e_n), \quad \text{in $B_{\delta},$}\\
\label{3a}& v_R(X - \frac{C_1}{R} e_n) \leq U(X), \quad\text{in $\overline{B}_1.$}
\end{align}}

\begin{proof} 
Estimates \eqref{4a} and \eqref{3a} are immediate consequences of \eqref{estvr} and Lemma \ref{elem}.

To obtain \eqref{2a}, notice that in view of  \eqref{estvr} and Lemma \ref{elem},
$$ v_R(X+ \frac{c_0}{R}e_n) \leq U(X) \quad \text{in $\{X\in B_1 : |x' | \geq 1/8, |x_n|  \leq \bar \delta\},$}$$ for some $c_0, \bar \delta$ small universal and $R$ large (with strict inequality on $F(v_R(X+ \frac{c_0}{R} e_n))$). 
Hence the estimate \eqref{2a} holds on  the set $\{X \in B_1 \setminus B_{1/4} : \sqrt{x_n^2 +z^2} \leq \bar \delta\}$ and we only need to prove it on the complement of this set. 

Again, from \eqref{estvr} and Lemma \ref{elem} we get that 
\begin{equation}\label{mon}  v_R(X+ \frac{c_0}{R}e_n) \leq U(X+ \frac{\bar C}{R}e_n) \quad \text{in $B_1,$}
\end{equation} for $\bar C$ large universal. 
From Lemma \ref{cor2} we know that 
$$U(x_n + \frac{\bar C}{R}, z) \leq (1+C \frac{\bar C}{R}) U(x_n, z),$$
as long as $\sqrt{x_n^2 + z^2}> \bar \delta$, with $C= C(\bar \delta)$ (and $R$ large). Combining this fact with \eqref{mon} we get 
$$v_R(X+ \frac{c_0}{R}e_n) \leq (1+\frac{C_0}{R})U(x_n,z),$$ on the desired set. 

\end{proof}

\end{document}